\documentclass{theoretics}
\usepackage{blkarray,bigstrut}
\usepackage{thm-restate} 
\newcommand{\RestateRemark}[1]{{\normalfont\bfseries #1}}
\newcommand{\RestateInit}[1]{\newcommand{#1}{}}
\newcommand{\RestateGo}[1]{\renewcommand{#1}{(Restated)}}

\newcommand{\remove}[1]{}

\def\F{{\mathbb{F}}}

\def\Z{{\mathbb{Z}}}
\def\N{{\mathbb{N}}}

\def\cF{{\cal F}}

\def\cL{{\cal L}}

\def\cV{{\mathcal V}}

\def\V{{\mathbf{V}}}

\def\bi{{\mathbf{i}}}
\def\bj{{\mathbf{j}}}

\def\E{{\mathbb E}}

\def\prob{{\mathbf{Pr}}}

\def\rank{\textsf{rank}}

\def\spanV{\textsf{span}}
\def\EVAL{\textnormal{\textsf{EVAL}}}
\def\Coeff{\textsf{Coeff}}

\newcommand{\eps}{\epsilon}

\title{Linear Hashing with \texorpdfstring{$\ell_\infty$}{l-infinity} guarantees and two-sided Kakeya bounds}

\ThCSauthor{Manik Dhar}{dmanik@mit.edu}[0009-0000-5570-7116]
\ThCSaffil{Department of Pure and Applied Mathematics, Massachusetts Institute of Technology}

\ThCSauthor{Zeev Dvir}{zdvir@princeton.edu}[0000-0002-4588-8440]
\ThCSaffil{Department of Computer Science and Department of Mathematics,
  Princeton University}

\ThCSshortnames{M.\ Dhar and Z.\ Dvir}

\ThCSshorttitle{Linear Hashing with \texorpdfstring{$\ell_\infty$}{l-infinity} guarantees and two-sided Kakeya bounds}

\ThCSthanks{Part of this work was done while the first author was a graduate student at Princeton University supported by NSF grant DMS-1953807. The second author is supported by NSF grant DMS-2246682. A preliminary version of this work appeared as an extended abstract in the Proceedings of FOCS 2022.}

\ThCSyear{2024}
\ThCSarticlenum{8}
\ThCSdoicreatedtrue
\ThCSreceived{Jul 3, 2023}
\ThCSaccepted{Feb 11, 2024}
\ThCSpublished{Mar 31, 2024}
\ThCSkeywords{Linear Hashing, Kakeya, Leftover Hash Lemma, Cryptography
}

\begin{document}

\maketitle

\begin{abstract}
We show that a randomly chosen linear map over a finite field gives a good hash function in the $\ell_\infty$ sense.  More concretely, consider a set $S \subset \F_q^n$ and a randomly chosen linear map $L : \F_q^n \to \F_q^t$ with $q^t$ taken to be sufficiently smaller than $ |S|$. Let $U_S$ denote a random variable distributed uniformly on $S$. Our main theorem shows that, with high probability over the choice of $L$, the random variable  $L(U_S)$ is close to uniform in the $\ell_\infty$ norm.  In other words, {\em every} element in the range $\F_q^t$ has about the same number of elements in $S$ mapped to it. This complements the widely-used Leftover Hash Lemma (LHL) which proves the analog statement under the statistical, or $\ell_1$, distance  (for a richer class of functions) as well as prior work on the expected largest 'bucket size' in linear hash functions \cite{LinearNoga}. By known bounds from the load balancing literature~\cite{MarAngeBalls}, our results are tight and show that linear functions  hash as well as truly random function up to a constant factor  in the entropy loss. Our proof leverages a connection between linear hashing and the finite field Kakeya problem and extends some of the tools developed in this area, in particular the polynomial method. 
\end{abstract} 


\section{Introduction}

Let $S \subset \{0,1\}^n$ be a set.  In many scenarios, one is interested in `hashing' the space $\{0,1\}^n$ into a smaller space so that the set $S$ (on which we may have little or no information) is mapped in a way that is close to uniform. Specifically, we may  need to find a function $H: \{0,1\}^n \to \{0,1\}^t$ so that the random variable $H(U_S)$ is close to the uniform distribution, where $U_S$ denotes a random variable distributed uniformly on the set $S$. An important parameter here is the `entropy-loss' given by $ \log_2|S| - t$. Clearly, this quantity has to be non negative, and, in practice, we would like it to be as small as possible. 

An important result in this area is the celebrated Leftover Hash Lemma (LHL) of Impagliazzo, Levin and Luby \cite{ILL} which asserts that the above scenario can be handled by choosing~$H$ at random from a family of universal hash functions (one in which for every $x \neq y$ the probability that $H(x) = H(y)$ is at most $2^{-t}$ over the choice of $H$).
\begin{lemma}[Leftover Hash Lemma \cite{ILL}]\label{lem-LHL}
Let $S \subset \{0,1\}^n$ and suppose $H : \{0,1\}^n \to \{0,1\}^t$  is chosen from a family of universal hash functions with $t \leq \log_2|S| - 2\log_2(1/\eps)$. Then the random variable\footnote{In the notation $(H,H(U_S)$ we assume that the function $H$ is represented by a string of bits of some fixed length.} $(H,H(U_S))$ is $\eps$-close to uniform in the $\ell_1$-norm.\footnote{Typically, the conclusion of the lemma is stated with respect to the statistical distance (or total variation distance) which is defined to be $1/2$ of the $\ell_1$ distance} 
\end{lemma}

A few comments about the LHL are in order. The first is that, using a standard averaging argument, the LHL implies that, for any given set, most choices of $H$ will be good, in the sense that $H(U_S)$ will be close to uniform in the $\ell_1$ distance.  It is also known that the {\em entropy loss} of the LHL, namely $2\log(1/\eps)$, is the smallest possible for any family of functions \cite{JaikumarShma2000}. Lastly, it is possible to generalize the LHL to handle arbitrary distributions of high min-entropy\footnote{A distribution has min-entropy at least $k$ if any output has probability at most $2^{-k}$.} (not just  those uniform on a set).  This follows from the fact that any distribution with min-entropy $k$ is a convex combination of `flat' distributions (those uniform on a set of size $2^k$).

A convenient choice of a universal family of hash functions is that given by all linear maps over the finite field  of two elements $\F_2$. That is, the LHL says that, if one picks a linear map $L :\F_2^n \to \F_2^t$ uniformly at random, then, with high probability over the choice of $L$, the random variable $L(U_S)$ will be close to uniform in the $\ell_1$-distance. Our main theorem shows that, with slightly larger entropy loss, one can give a stronger guarantee on the output, stated in $\ell_\infty$ distance to uniform. A reason to consider linear maps is their simplicity and ease of implementation (only requiring very basic bit operations) for applications. Since the full statement of the theorem is quite technical (stemming from our attempts to optimize the various constants) we start by giving an informal statement. The full statements of our results (also for other larger finite fields) are given in Section~\ref{sec-formal}.

\begin{theorem}[Main theorem (informal)]\label{thm-main-informal}
Let $S \subset \F_2^n$ and let $t = \log_2|S| - O(\log_2(\log_2|S|/\tau\delta))$. Then, a $(1 - \delta)$-fraction of all linear maps $L :\F_2^n \to \F_2^t$ are such that $L(U_S)$ is $\tau 2^{-t}$-close to uniform in the $\ell_\infty$ norm. That is, for all $y \in \F_2^t$ we have  $$ | \prob[ L(U_S)=y ] - 2^{-t}| \leq \tau 2^{-t}.$$
\end{theorem}
An equivalent way to state this theorem comes from the observation that the set of elements in $\F_2^n$ mapping to a particular $y \in \F_2^t$ is always of the form $a_y + U$, where $U$ is the kernel of the mapping $L$ and $a_y \in \F_2^n$ is some shift.  In this view, the theorem says that most $(n-t)$-dimensional subspaces $U \subset \F_2^n$ are such that {\em all} of their shifts intersect $S$ in about the same number of points (up to a multiplicative factor of $1 \pm \tau$). We devote Section~\ref{sec-furstenberg} to a more detailed treatment of this view, which will be the one used in the proof. The question of bounding the maximal `bucket size' (all elements mapping to a single $y \in \F_2^t$) in a random  linear hash function was previously studied and  we compare our results to the state-of-the-art in this area (\cite{LinearNoga}) in Section~\ref{sec-formal} after the formal statement of our results.

Our choice of the letter $\tau$ instead of $\eps$ as in the LHL is not accidental and is meant to highlight the fact that, in the $\ell_\infty$ setting, we can take $\tau$ to be  greater than~$1$. When $\tau < 1$ the conclusion of our theorem, namely that $L(U_S)$ is $\tau/2^{-t}$-close to uniform in $\ell_\infty$, implies that $L(U_S)$ is also $\tau$-close to uniform in $\ell_1$. However, our theorem is still meaningful when $\tau > 1$, even though it says nothing about $\ell_1$ distance. The advantage of taking $\tau$ to be large comes from the fact that it can reduce our entropy loss (this can be done up to a point, as is stated in the formal theorem statement below). To give an example of a scenario in which we can take large~$\tau$, consider the case where the linear map $L$ is used to derive a key for a digital signature scheme. We would like the key $L(U_S)$ to be close to uniform since we know that a uniform key prevents the adversary from producing a forgery with more than negligible probability. However, if we apply our theorem with large $\tau$ (say polynomial in $n$) we get that the probability of producing a forgery may increase by at most a factor of $1+\tau$ which still results in negligible probability of forgery. More generally, the case of large $\tau$ is relevant whenever we only care about events of small probability staying small. Another paper that focuses on these aspects of the LHL (that is, when we only care about low probability events) is \cite{LHLRevisited}.

The need for $\ell_\infty$ guarantees for hashing appears in many places in the literature. For example, in  Cryptography,  in the context of key generation for local data storage~\cite{dictionaryCrypt} and batch verifying zero-knowledge proofs~\cite{StatZeroCrypto} and in Computational  in the context of  uniformly generating a solution to NP-search problems (see Section 6.2.4.2 in \cite{goldreich_2008}). It is possible to guarantee~$\ell_\infty$ hashing by using a larger and more complex classes of functions, for example high degree polynomials over a large finite field~\cite{ALONrandomized}. For the applications in \cite{dictionaryCrypt} and \cite{goldreich_2008} our results allows one to use linear maps instead of polynomials, hence simplifying the proofs.

Our proofs leverage a connection between linear hashing and finite field Furstenberg sets (which generalize Kakeya sets). A $k$-dimensional Furstenberg set $S \subset \F_q^n$ is a set which has a large intersection with a $k$-flat ($k$-dimensional affine subspaces) in each direction. That is, for any $k$-dimensional subspace $U \subset \F_q^n$ there is a shift $s(U)$ such that the affine subspace $s(U) + U$ has a large intersection with the set $S$. The goal in this area is to prove lower bounds on the size of such sets. Surprisingly, such lower bounds play a role in explicit constructions of seeded extractors \cite{DW08,DKSS13} which are randomness efficient variants of the LHL. However, the connection between Furstenberg sets and linear hashing we leverage in this paper  is  {\em unrelated} to the work on extractors mentioned above and is of a completely different nature. This connection was first observed in \cite{DDL-2} and was used there to improve the best lower bounds on Furstenberg sets. Our work relies heavily on the methods developed in \cite{DDL-2} (as well as other papers) and extends them in several respects. We devote Section~\ref{sec-furstenberg} to a more complete discussion of this connection and, in particular, to explaining the phrase `two-sided Kakeya bounds' from the title of the paper.

\paragraph{Acknowledgments:} We are  grateful to Or Ordentlich, Oded Regev and Barak Weiss for comments that led us to pursue this line of work. Their interest in theorems of this kind  arose from trying to strengthen their breakthrough \cite{ORW22} on lattice coverings, which uses the two dimensional Kakeya bounds of \cite{KLSS2011}. (A new paper by the same group of authors, using the results of the current paper, is in preparation.) We are also grateful to the reviewers for their suggestions, especially for pointing out that Theorem~\ref{thm-smoothSlice} also follows from our arguments.

\paragraph{Paper organization:} The rest of the paper is organized as follows. In Section~\ref{sec-formal} we state our main theorems formally. In Section~\ref{sec-related} we discuss the tightness of our results, compare them to prior work, and discuss possible generalizations.
In Section~\ref{sec-furstenberg} we discuss the connection to the theory of Furstenberg/Kakeya sets and introduce notations and definitions that will be used in the proofs. In Section~\ref{sec-overview} we give a high level overview of the proof. Section~\ref{sec-proofs}  contains the  proofs of our  main theorems with a  lemma, giving an improved bound on Furstenberg sets, proved in Section~\ref{sec-furstenberg-proof}.

\section{Formal statement of our results}\label{sec-formal}

This section contains four variants of our main result. The four cases correspond to the distinction between large finite fields and $\F_2$ and between arbitrary  $\tau$  and the special case $\tau > 1$ (in which we can get slightly better constants). We begin with the statement for large finite field and arbitrary $\tau$.

\begin{theorem}\label{thm-mainhash}
Let  $n \geq 5$ and let $S \subset \F_q^n$ be a set.  Let $\tau > 0$ be a real number and $\delta\in (0,1)$ such that 
$$q \geq 32\max\left(\frac{n(1+\tau)}{(\tau\delta)^2},n\right).$$ Suppose $q^r < |S| \leq q^{r+1}$ for some $4 \leq r \leq n-1$ and let $t = r-3$.  Then a $(1-\delta)$-fraction of all surjective linear maps $L : \F_q^n \to \F_q^t$ are such that  $L(U_S)$ is $\tau/q^t$-close to the uniform distribution in the $\ell_\infty$ norm.
\end{theorem}

Notice that, in the setting above, the entropy loss, when measured in $\F_q$-dimension is at most 4. The restriction to the case of surjective linear maps is natural as these are maps that do not `lose' entropy unnecessarily  (one can consider all linear maps by increasing $\delta$ slightly). 

The above theorem can be used to derive similar results for small fields, by treating blocks of coordinates as representing elements in an extensions field. We do this for every possible choice of basis to ensure that our theorem works for all surjective $\F_2$-linear maps. We only treat the case of $\F_2$ as this is the field  most commonly used in applications (the same proof strategy will work for any  finite field).

\RestateInit{\restatebinary}
\begin{restatable}{theorem}{theobinary}
    \label{thm-binary}\RestateRemark{\restatebinary}
	Let $S \subset \F_2^n$ be such that $|S| > 2^{20}\max(n^4(1+\tau)^4/(\tau\delta)^8,n^4)$ and let $n,\tau,\delta$ satisfy $n\ge 5 \lceil\log_2(\max(n(1+\tau)/(\tau\delta)^2,n))\rceil+25$. Then there exists a natural number 
	$$t\ge  \log_2|S| - 4\log_2\left(\max\left(\frac{n(1+\tau)}{(\tau\delta)^2},n\right)\right)-20,$$ 
	such that a $(1-\delta)$-fraction of all surjective linear maps $L : \F_2^n \to \F_2^t$ are such  that $L(U_S)$ is $\tau 2^{-t}$-close to uniform in the $\ell_\infty$ norm.
\end{restatable}

When $|S|$ is small we can improve the previous theorem by replacing the $n$ in the entropy loss by $\log_2|S|$. This is achieved using the following simple lemma, which allows us to first hash~$S$ into a universe of size roughly $|S|^2$ without any collisions.

\begin{lemma}\label{lem-birthHash}
Let $S\subset \F_2^n$ and
$$t \ge \log_2(|S|(|S|-1)/2\delta).$$
Then, at least a $(1-\delta)$-fraction of all surjective linear maps $L: \F_2^n\to \F_2^t$ map $S$ injectively into $\F_2^t$.
\end{lemma}
\begin{proof}
    As surjective linear maps are a universal family of hash functions we have,
    $$\prob[L(x)=L(y)]\le 1/2^t\le \delta \frac{2}{|S|(|S|-1)}$$
    for a random surjective linear map $L:\F_2^n\to \F_2^t$ and $x,y\in S,x\ne y$.
    By applying the union bound we see the probability that $L$ is not injective is upper bounded by $\delta$.
\end{proof}

Applying the above lemma followed by Theorem~\ref{thm-binary} immediately leads to a concrete instance of Theorem~\ref{thm-main-informal}.

\begin{theorem}\label{thm-BetBinary}
	Let $S \subset \F_2^n,\tau,\delta\in (0,1)$ and $m=\log_2(|S|(|S|-1)/\delta)$ be such that 
	\begin{align}
	    |S| &> 2^{20}m\max(2^8(1+\tau)^4/(\tau\delta)^8,1)\label{eq-co1}\\
	    m&\ge 5 \log_2(m\max(4(1+\tau)/(\tau\delta)^2,1))+25\label{eq-co2},
	\end{align}
	then there exists a natural number 
	$$t\ge  \log_2|S| - 4\log_2\left(m\max\left(\frac{4(1+\tau)}{(\tau\delta)^2},1\right)\right)-20,$$ 
	such that a $(1-\delta)$-fraction of all surjective linear maps $L : \F_2^n \to \F_2^t$ are such  that $L(U_S)$ is $\tau 2^{-t}$-close to uniform in the $\ell_\infty$ norm.
\end{theorem}
\begin{proof}
    We apply Lemma~\ref{lem-birthHash} for $\delta/2$ and linear maps from $\F_2^n\to\F_2^m$ followed by applying Theorem~\ref{thm-binary} for $\delta/2$ and linear maps from $\F_2^m\to \F_2^t$.
\end{proof}

The conditions \eqref{eq-co1} and \eqref{eq-co2} are not very restrictive. In the setting $\tau=\delta=1/n^C$ for some constant $C$ conditions \eqref{eq-co1} and \eqref{eq-co2} are satisfied for $|S|\ge n^{C'}$ where $C'$ only depends on $C$.

An interesting setting of parameters for Theorem~\ref{thm-BetBinary} is that of $\tau = 1/\delta^2$. In this case (when~$\delta$ is sufficiently small), the two terms in the `max' function above are about the same and we get an entropy loss of $4\log_2(4\log_2(|S|(|S|-1)/\delta))$. With this entropy loss, we get that the output $L(U_S)$ is $(1/\delta)^2 \cdot 2^{-t}$ close to uniform in the $\ell_\infty$ norm. Or, in other words, for $(1-\delta)$-fraction of linear maps $L$,  the probability of any event under $L(U_S)$ is at most  a multiplicative factor of $1/\delta^2$ larger than its probability under the uniform distribution. In this setting \eqref{eq-co1} and \eqref{eq-co2} are satisfied by ensuring $|S|$ is larger than some fixed universal constant.

\paragraph{Improvements when $\tau>1$: } In this setting, we can improve the constant in the above two theorems slightly. We start with the case of large finite field. In the following theorem, the bound on the size of $q$ does not contain the constant 32 appearing in Theorem~\ref{thm-mainhash}. The dependence of $q$ on $\tau$ changes from $\frac{1+\tau}{\tau^2}$ to $\frac{1+\tau}{(\tau-\sqrt{\tau})^2}$ which are asymptotically the same when $\tau$ grows. Hence, when  $\tau$ is sufficiently large, the saving in $q$ is roughly a factor of $32$. The price we pay for this improvement is  the need for  $n$ to be at least $ 20$ (as opposed to 5) and an upper bound $\delta < 1/10$.

\begin{theorem}\label{thm-imphash}
Let  $n \geq 20$ and let $S \subset \F_q^n$ be a set.  Let $\tau > 1$ be a real number and $\delta\in (0,1/10)$ such that 
$$q \geq \max\left(n\frac{1+\tau}{(\tau-\sqrt{\tau})^2\delta^2},n\right).$$
Suppose $q^r < |S| \leq q^{r+1}$ for some $4 \leq r \leq n-1$ and let $t = r-3$.  Then a $(1-\delta)$-fraction of all surjective linear maps $L : \F_q^n \to \F_q^t$ are such that $L(U_S)$ is $\tau/q^t$-close to the uniform distribution in the $\ell_\infty$ norm.
\end{theorem}

As before, this can be used to  prove a version over $\F_2$ for large $\tau$ with improved constants.

\begin{theorem}\label{thm-impbinary}
	Let $S \subset \F_2^n$ and let $\delta\le 1/10$, $\tau >1$ be such that $|S| > \max(n^4(1+\tau)^4/((\tau-\sqrt{\tau})\delta)^8,n^4)$ and $n,\tau,\delta$ satisfy $n\ge 20\lceil\log_2(\max(n(1+\tau)/((\tau-\sqrt{\tau})\delta)^2,n))\rceil$. Then there exists a natural number 
	$$t\ge  \log_2|S| - 4\log_2\left(\max\left(n\frac{1+\tau}{(\tau-\sqrt{\tau})^2\delta^2},n\right)\right),$$ such that a $(1-\delta)$-fraction of all surjective linear maps $L : \F_2^n \to \F_2^t$ have the property that $L(U_S)$ is $\tau 2^{-t}$-close to uniform in the $\ell_\infty$ norm.
\end{theorem}

We can again use Lemma~\ref{lem-birthHash} to improve the entropy loss in the previous theorem.

\begin{theorem}\label{thm-Betimpbinary}
	Let $S \subset \F_2^n$ and let $\delta\le 1/10$, $\tau >1$ and $m=\log_2(|S|(|S|-1)/\delta)$ be such that 
	\begin{align*}
	    |S| &> m^4\max(2^8(1+\tau)^4/((\tau-\sqrt{\tau})\delta)^8,1)\\
	    m&\ge 20\lceil\log_2(m\max(4(1+\tau)/((\tau-\sqrt{\tau})\delta)^2,1))\rceil\,.
	\end{align*} Then there exists a natural number 
	$$t\ge  \log_2|S| - 4\log_2\left(m\max\left(\frac{4(1+\tau)}{(\tau-\sqrt{\tau})^2\delta^2},1\right)\right),$$ such that a $(1-\delta)$-fraction of all surjective linear maps $L : \F_2^n \to \F_2^t$ have the property that $L(U_S)$ is $\tau 2^{-t}$-close to uniform in the $\ell_\infty$ norm.
\end{theorem}

\subsection{Some comments} \label{sec-related}

\paragraph{Tightness of our results:} 

It is natural to ask whether our results are tight. Fixing the parameter $\delta$ to be constant for the sake of simplicity, can we possibly improve on the entropy loss stated in  Theorem~\ref{thm-BetBinary}? The answer is a resounding No! Even  for a truly random function, the results of \cite{MarAngeBalls} show that we need an entropy loss of at least $\log_2(\log_2|S|/\tau^2)$ (up to a additive constant) to achieve the conclusion of Theorem~\ref{thm-BetBinary}.  Hence, up to a reasonably small constant factor (of about 32), linear functions hash as well as random functions.

\paragraph{Prior results on linear hash functions:}
Properties of random linear hashes with respect to the $\ell_\infty$ norm have been studied in earlier works~\cite{CARTERWEGMANS,mehlhorn1984randomized,LinearNoga} with  \cite{LinearNoga} being the state-of-the-art. The results in this area are typically stated as upper bounds on the expected   'maximal bucket size' (that is, the maximum size of $L^{-1}(y)$ over all $y \in \F_q^t$). We will see that earlier results only give bounds for $\tau\gg 1$ (as far as we know, our paper is the first to give $\ell_\infty$ guarantees for small $\tau$). 

Theorem 5  of \cite{LinearNoga} is the most relevant to this work and shows that, when $\log_2|S| - t = \log_2(t)$ the expected maximal bucket size is $O(t\log_2(t))$. A Markov argument shows then, that, with probability at least $1 - \delta$, the maximal bucket size is at most $O(t\log_2(t)/\delta)$ which is a factor of $\log_2(t)/\delta$ larger than the trivial bound of $|S|/2^t = t$. Note that $\log_2(t)/\delta\gg 1$. 

Theorem~\ref{thm-BetBinary} for small $\tau$ shows that when $\log_2|S| - t \approx  O(\log_2(\log_2(|S|^2/\delta)(\tau\delta)^{-2}))$, the maximal bucket size will be at most a factor of  $1+\tau$ larger than the trivial bound of $|S|/2^t$ with probability $1-\delta$ over the choice of the linear function. Hence, the results of \cite{LinearNoga} deal with the case of smaller entropy loss  ($\log_2(t)\approx \log_2(\log_2(|S|))$ instead of  $O(\log_2(\log_2(|S|^2/\delta)(\tau\delta)^{-2}))$ ) but are a multiplicative factor of $\log_2(t)/\delta\gg 1$ away from uniform instead of $\tau+1$ which can be made arbitrarily close to $1$ (by reducing $\tau$ and increasing the entropy loss).

We can also make comparisons in the regime of large $\tau$. As stated earlier for $\tau=1/\delta^2$ Theorem~\ref{thm-BetBinary} shows that the maximal bucket size will be at most a factor of $1+1/\delta^2$ larger than the trivial bound of $|S|/2^t$ with probability $1-\delta$ over the choice of the linear function. In this setting for $\delta\gg 1/\log_2(t)$, we lose a constant factor in the entropy loss ($\log_2\log_2|S|$ in \cite{LinearNoga} and $O(\log_2\log_2|S|)$ for our result) and gain in the bucket size bound ($\log_2(t)\delta$ times $|S|/2^t$ in \cite{LinearNoga} and $1+1/\delta^2$ times $|S|/2^t$ for our result). Although it should be noted that the results in \cite{LinearNoga} are incomparable in the sense that they compute the expected value of the bucket size while our results only give bounds on the bucket size with high probability.

\paragraph{Other families of universal hash functions:} In this section we look at whether our results can hold for other universal families of hash functions.

We first show that our results can not hold for all families of universal hash functions by means of an example. The family we will consider is linear maps from $\F_{q^2}^2$ to $\F_{q^2}$ which do form a universal family. We will show that that known results from \cite{LinearNoga} prove that this family needs at least an entropy loss of $\Omega(\log_2|S|)$ to get the distance guarantees of Theorem~\ref{thm-BetBinary}. This also shows that we need high dimensionality to get good linear hash function over large fields.

Theorem 8 of \cite{LinearNoga} proves that for any finite field $\F_{q^2}$ where $q$ is a prime power if we consider the set of linear maps from $\F_{q^2}^2$ to $\F_{q^2}$ then there exists a set $S_0$ of size $q^2$ such that for every linear map the maximal bucket size is at least $q$.

This implies that for any $S'_0$ of size $q^{2+\eta},\eta<1$ which contains $S_0$, every linear map $L:\F_{q^2}^2\to \F_{q^2}$ will have a maximal bucket size of at least $q$.  In other words $L(U_{S'_0})$ will be at least $1/q^{1+\eta}\gg C/q^2$ away from uniform in $\ell_\infty$ distance. Equivalently, even for an entropy loss of $\eta\log_2(q)=\Omega(\log_2|S'_0|)\gg O(\log_2\log_2|S|)$, linear maps from $L:\F_{q^2}^2\to \F_{q^2}$ do not guarantee that the image $L(U_{S'_0})$ will be $C/q^2$ close to uniform for any fixed constant $C$. This also means that we need at least an entropy loss of $\Omega(\log_2|S|)$ to get the distance guarantees of Theorem~\ref{thm-BetBinary}.

Other families of universal hash function could still achieve the guarantees of Theorem~\ref{thm-BetBinary}. In particular, for a prime $p$ consider the family of hash functions $h_{a,b}:\{0,1,\hdots,p-1\}\to\{0,\hdots,m-1\}$ for $a\in \{1,\hdots,p-1\},b\in\{0,\hdots,p-1\}$ defined as $h_{a,b}(x)= (ax+b \mod p) \mod m$. From \cite{CARTERWEGMANS}, we know that this family is universal. By following the framework in Section~\ref{sec-furstenberg}, it can be checked that proving $\ell_\infty$-guarantees for this family is a generalization of the notoriously difficult Arithmetic Kakeya problem~\cite{green2019arithmetic}.

\paragraph{The case of high min-entropy:} As was mentioned before, The LHL holds not just for `flat' distributions of the form $U_S$, but for any distribution with high min-entropy.  This more general version can be derived easily from the LHL for sets using a convex combination argument. As far as we can tell, this argument fails in the case of $\ell_\infty$ and so we cannot automatically derive a min-entropy analog of our results. While we do believe that our proof techniques could be made to handle this more general case (e.g., as is the case in  \cite{DDL-2}), we leave it for future work.

\section{Connection to prior work on Kakeya and Furstenberg sets}\label{sec-furstenberg}

In this section we will explain the connection between Theorem~\ref{thm-mainhash} and the finite field Kakeya-Furstenberg problem. Along the way we will introduce notations and definitions that will be used later on in the proofs. 

We will now describe an equivalent formulation of Theorem~\ref{thm-mainhash} in terms of the kernel of the linear map $L : \F_q^n \to \F_q^t$ appearing in the theorem. This will allow us to highlight its connection to the finite field Kakeya problem. To do so, we introduce some notations. For $1 \leq k \leq n$ we denote by $\cL_k(\F_q^n)$ the set of $k$-dimensional flats in $\F_q^n$ and by $\cL_k^*(\F_q^n)$ the set of $k$-dimensional subspaces (flats passing through the origin). Let $S \subset \F_q^n$ be a set. For $k \in [n]$, we denote by $$E_k(S) =  |S|/q^{n-k}$$ the expectation of $|R \cap S|$ with $R$ chosen uniformly in $\cL_k(\F_q^n)$. When $S$ is clear from the context we omit it and simply write $E_k$.
\begin{definition}\label{def-eps-balanced}
We say that $R\in \cL_k(\F_q^n)$ is $\tau$-balanced with respect to a set $S \subset \F_q^n$ if we have:
$$ \left| |R\cap S| - E_k(S) \right| \leq \tau \cdot E_k(S). $$ Otherwise, we say that $R$ is $\tau$-unbalanced with respect to $S$.
\end{definition}

\begin{definition}\label{def-shift-balanced}
We say that $A\in \cL_k^*(\F_q^n)$ is $\tau$-shift-balanced with respect to $S$  if, for all $a \in \F_q^n$, the flat $R = A+a$ is $\tau$-balanced with respect to $S$.
\end{definition}

Notice that if $A\in \cL_k^*(\F_q^n)$ is $\tau$-shift-balanced with respect to $S$ and $A'\in \cL_{k'}^*(\F_q^n)$ contains $A$ (with $k' > k$) then $A'$ is also $\tau$-shift-balanced with respect to $S$.

We will now express Theorem~\ref{thm-mainhash} using this new notation. Suppose $L : \F_q^n \to \F_q^t$ is an onto linear map  and let $A = \ker(L)$ be its $k=n-t$ dimensional kernel. Notice that, for each $y \in \F_q^t$, $$\prob[ L(U_S)=y ] = \frac{|(A+a)\cap S|}{|S|},$$ for some $a\in \F_q^n$ for which $L(a)=y$. Therefore, $$ | \prob[ L(U_S)=y ] - q^{-t}| \leq \tau q^{-t},$$ if and only if $A+a$ is $\tau$-balanced with respect to $S$. Hence, Theorem~\ref{thm-mainhash} is equivalent to the following theorem.

\begin{theorem}\label{thm-existsSB}
Let  $n \geq 5$ and let $S \subset \F_q^n$ be a set such that $|S| > q^4$.  Let $\tau > 0, \delta\in (0,1)$ be a real number such that $q \geq 32 \max(n(1+\tau)/(\tau\delta)^2,n)$. Let $4 \leq r \leq n-1$ be an integer such that $q^r < |S| \leq q^{r+1}$ and let $k =n - r + 3$.   Then a $(1-\delta)$-fraction of all subspaces in $\cL_k^*(\F_q^n)$ are $\tau$-shift-balanced with respect to $S$.
\end{theorem}

The above statement can also be read as saying that for a dimension $k$ such that $q^k|S|>q^{n+3}$ then most $k$ dimensional subspaces are going to shift-balanced. We can improve this by requiring a larger field size. While this statement is not going to help improve our hashing result, we believe it could have other applications.

\begin{theorem}\label{thm-smoothSlice}
Let  $n \geq 5, \eta\in (0,1]$ and let $S \subset \F_q^n$ be a set such that $|S| > q^4$. Then there exists a constant $C_\eta>0$ depending only on $\eta$ such that for any $\tau > 0, \delta\in (0,1)$ satisfying $q^{\eta} \geq  C_\eta\max(n(1+\tau)/(\tau\delta)^2,n)$ and any integer $k$ satisfying $q^k|S| > q^{n+2+\eta}$ we have that a $(1-\delta)$-fraction of all subspaces in $\cL_k^*(\F_q^n)$ are $\tau$-shift-balanced with respect to $S$.
\end{theorem}

We see that $\eta$ can be made arbitrarily small, as long as the field is large enough.

We now take a moment to explain the expression `two-sided Kakeya bounds' from the title and the connection to prior work on Kakeya sets. A Kakeya set in $\F_q^n$ is a set containing a line in each direction. The main question, asked by Wolff in \cite{Wolff99}, is to lower bound the size of such sets. This question has now been completely resolved in the series of papers \cite{Dvir08, DKSS13, BukhChao21}. We will be mostly interested in the high dimensional variants of this problem, asking about sets containing $k$-dimensional flats in all directions, or more generally, sets that have large intersection with a flat in each direction (these are called Furstenberg sets). These type of questions have been also studied extensively, with tight bounds obtained in some cases \cite{EOT10,KLSS2011, EE16, DDL-1,DDL-2}.

 We start by recalling some definitions from that domain. 
 
\begin{definition}[$m$-rich flats]
We call a flat $R \in \cL_k(\F_q^n)$ $m$-rich with respect to a set $S \subset \F_q^n$ if $|R \cap S| \geq m$. 
\end{definition}

\begin{definition}[$(k,m,\beta)$-Furstenberg sets]
We call a set $K \subset \F_q^n$ a $(k,m,\beta)$-Furstenberg set if $K$ has an $m$-rich $k$-flat for at least a $\beta$ fraction of directions. That is, for at least a $\beta$-fraction of all $A \in \cL_k^*(\F_q^n)$ there exists $a \in \F_q^n$ so that $a + A$ is $m$-rich with respect to $K$. 
\end{definition} 

$(k,q^k,1)$-Furstenberg sets are also called Kakeya sets.
Prior works on Kakeya/Furstenberg sets were focused on giving {\em lower bounds} on the size of $(k,m, 1)$-Furstenberg sets. For example, in \cite{DDL-2}, it was shown that, if $S$ is a $(k,m,1)$-Furstenberg set then $|S| > (1-\eps)mq^{n-k}$, assuming $q$ is sufficiently large as a function of $n$ and $\eps$ (in particular, $q$ has to be exponential in $n$). Notice that this is the best possible since any set of size $mq^{n-k}$ is $(k,m,1)$-Furstenberg. Stated in the counter-positive direction, this theorem shows that: If 
 \begin{equation}\label{eq-sizeS}
 	|S| \leq (1-\eps)mq^{n-k}
 \end{equation}
then there exists a $k$-dimensional subspace $R$ such that all shifts of $R$ have less than $m$-points in common with $S$.  Notice that (\ref{eq-sizeS}) gives us that $$E_k(S) \leq (1-\eps)m.$$ So, what we discover is that, the results in \cite{DDL-2} simply say that, for every $S$, there is a subspace $R$ such that all shifts of $R$ have intersection with $S$ that is not much larger from the expectation $E_k$. Hence, Theorem~\ref{thm-existsSB} can be viewed as a two-sided generalization of this statement by showing that, in fact, there exists $R$ such that all shifts of $R$ have roughly the expected intersection with $S$.

\section{Proof Overview}\label{sec-overview}
We now give a short sketch of the proof of  Theorem~\ref{thm-mainhash}.  The proof of Theorem~\ref{thm-imphash} (the case of $\tau > 1$)  is essentially the same as the proof of Theorem~\ref{thm-mainhash} with a different setting of a single parameter and so we will not discuss it here. We will also not discuss the two theorems dealing with the case of $\F_2$ as they will follow from the large field case by a simple encoding argument.

As discussed in Section~\ref{sec-furstenberg}, Theorem~\ref{thm-mainhash} is equivalent to Theorem~\ref{thm-existsSB} which is stated in the language of shift-balanced sub-spaces.  Given a set $S \subset \F_q^n$, the theorem claims that there are many sub-spaces $A\in \cL_k^*(\F_q^n)$ that are $\tau$-shift-balanced. Let us instead try and prove the easier claim that {\em there exists} at least one such subspace. We will prove this by contradiction. Suppose there are no $\tau$-shift balanced sub-spaces $A$. Then, for each $A\in \cL_k^*(\F_q^n)$ we can find a shift $f(A) \in \F_q^n$ so that the flat $T_A = f(A) + A$ is $\tau$-unbalanced.

At a very high level, the contradiction will follow by combining the following three statements:
\begin{itemize}
    \item (Concentration Statement) A random $k-2$ flat is  $\tau/2$-balanced with high probability.
    \item (Anti concentration statement) If $T$ is a $\tau$-unbalanced $k$-flat and $R$ is a randomly chosen $(k-2)$-flat in $T$ then $R$ is $\tau/2$-unbalanced with high probability.
    \item (Kakeya statement) Given a collection of $k$-flats, $T_A$, one in each direction $A \in \cL_k^*(\F_q^n)$. A randomly chosen $(k-2)$-flat in a randomly chosen $T_A$ `behaves like' a truly random  $(k-2)$-flat.
\end{itemize}

Before we discuss the proofs of these statements, let us see how they can be combined to derive a contradiction. Consider the distribution on $(k-2)$-flats obtained by sampling $A \in \cL_k^*(\F_q^n)$ uniformly at random and then choosing a random $(k-2)$-flat $R$ inside $f(A) + A$, where $f(A)$ is defined above so that $f(A)+A $ is $\tau$-unbalanced.  By the anti-concentration statement, this distribution outputs a $\tau/2$-unbalanced $R$ with high probability. Now, from the Kakeya statement we get that this should (in some way) also be the behaviour of a truly random $(k-2)$-flat, contradicting the concentration statement. This is essentially the structure of the proof, with the `behaves like' portion of the Kakeya statement replaced by a quantitative bound on the probability of landing in a given small set (the set of unbalanced $(k-2)$ flats).

Let us now discuss the proofs of the three statements. The first two  (concentration and anti-concentration), follow easily from Chebyshev's inequality and pair-wise independence and so we will only be concerned with the proof of the third one.  We can generalize the Kekeya statement as follows, given a collection of $k$-flats $T_A$, one in each direction $A$, what can be said about the distribution of a random $r$-flat $R$ in a random $T_A$ where we allow $r$ to be in the range $\{0,1,\ldots,k\}$. To recover the original (one dimensional) Kakeya problem all we have to do is set $k=1$ and $r=0$. Now, we are asking about the distribution of a random point $R$ on a  line~$T_A$ chosen so that its direction is uniformly random and its shift is arbitrary. The finite field Kakeya conjecture (proved in \cite{Dvir08}) says that  the distribution of $R$ has large support. In \cite{DW08,DKSS13}, motivated by applications to extractors, it was shown that, in fact, the distribution of $R$ has high min-entropy. These results can be easily `lifted up' to the case where $k>1$ and $r=k-1$ but, alas, the known (and tight) quantitative bounds on the min entropy are not sufficient for our purposes. Specifically, it is possible for the distribution of $R$ in this case to be contained in a set of density $2^{-n}$ inside $\F_q^n$, which is much too small for our purposes. This motivates us to take $r=k-2$, which reduces to understanding the case of $k=2$ and $r=0$. That is, given a family of $2$-flats $T_A$, one in each direction, what can be said about the behaviour of a random point $R$ on a random $T_A$?  Luckily, in this case, the results of \cite{KLSS2011, DDL-2} can be used to show that the distribution of $R$ has support with density approaching one.

To prove our theorem we need to extend the results of \cite{KLSS2011, DDL-2} in several ways, including going from support size to min entropy, reducing the field size from exponential to polynomial and handling the case of `many' directions instead of `all' (which corresponds to the parameter~$\delta$ being less than one). The required lemma is stated below and proved in Section~\ref{sec-furstenberg-proof}.

\RestateInit{\restatefur}
\begin{restatable}[Furstenberg lemma]{lemma}{lemfurii}
  \label{lem-fur2-intro}\RestateRemark{\restatefur}
For any $\gamma,\beta \in [0,1], n\in \N$, $q$ a prime power every $(2,\gamma q^2,\beta)$-Furstenberg set $K\subseteq \F_q^n$ has size at least,
$$|K|\ge \beta\gamma^n q^n\left(1+\frac{1}{q}\right)^{-n}.$$
\end{restatable}

We note that this lemma has been proven in \cite{ORW22} with a slightly worse lower bound of $\beta\gamma^n q^n\left(1+\frac{2}{q}\right)^{-n}$. This is enough to prove Theorem~\ref{thm-mainhash} leading to slightly worse constants in the field size requirement and hence the entropy loss of the theorem. The proof in \cite{ORW22} uses a combinatorial reduction to reduce the case of arbitrary $\beta$ to constant $\beta$. We give a new argument to prove this lemma directly.

Our proof of this lemma follows along the lines of prior works in this area and uses the polynomial method. One important ingredient is a new variant of the celebrated Schwartz-Zippel lemma which allows us to improve the dependence on $\beta$ above from $\beta^n$ to just $\beta$ (See Corollary~\ref{cor:goodMono}). We believe this lemma could have  applications in other situations where the polynomial method is used. For instance in a later work~\cite{dhar2022maximal} extensions of these arguments are used to prove maximal Kakeya bounds in the general setting of the integers modulo a composite number.

\section{Proof of Theorems~\ref{thm-mainhash} and \ref{thm-imphash}}\label{sec-proofs}
We prove Theorems~\ref{thm-mainhash} and \ref{thm-imphash} by contradiction. We will prove the equivalent versions of the theorems stated using $\tau$-shift-balanced subspaces (Theorem~\ref{thm-existsSB} and similarly for Theorem~\ref{thm-imphash} even though it was not stated separately). The proof of Theorem~\ref{thm-smoothSlice} is nearly identical, we give the modifications at the end of this section. 

\begin{proof}
Suppose the Theorems are not true. Then there exists a function with parameters as in the Theorems:
$$f:\cL^*_k(\F^n_q)\rightarrow \F_q^n $$
such that, for a $\delta$ fraction of $A\in \cL^*_k(\F^n_q)$, the flat $f(A)+A$ is $\tau$-unbalanced with respect to $S$. Notice that $f(A)$ can be taken to be any point on the flat $f(A)+A$ (the choice doesn't matter for this proof). 

For a real number $\sigma > 0$, let $$B^\sigma_{k-2}\subset \cL_{k-2}(\F^n_q)$$ denote the set of $(k-2)$-flats that are $\sigma$-unbalanced with respect to $S$. We will eventually set $\sigma$ to one of two values: To prove Theorem~\ref{thm-mainhash} we will set $\sigma = \tau/2$ and, to prove Theorem~\ref{thm-imphash} (when $\tau>1$) we will set $\sigma = \sqrt{\tau}$. Notice that, in both cases, we have $\tau - \sigma > 0$.

For a $k$-flat $T\in \cL_k(\F^n_q)$ we let $\cL_{k-2}(T)$ be the set of $(k-2)$-flats contained in $T$ and  let $$B^\sigma_{k-2}(T)=B^\sigma_{k-2}\cap \cL_{k-2}(T)$$ denote the set of $\sigma$-unbalanced $(k-2)$-flats with respect to $S$ that are contained in $T$.

Notice first that, by our assumption on $r$, we have
\begin{equation}\label{eq-boundk}
	 4 \leq k \leq n-1
\end{equation}

Throughout, we use $\V(X)$ to refer to the variance of a random variable $X$.

Our first claim shows that a random $(k-2)$ flat is balanced with high probability This gives the `concentration' part of the argument laid out in the proof overview.
\begin{claim}\label{cla-Bsmall}
If $R$ is chosen uniformly in $\cL_{k-2}(\F_q^n)$ then 
\[ \prob[ R \in B^\sigma_{k-2} ] \leq \frac{1}{\sigma^2 q}. \]
\end{claim}
\begin{proof}
	Since $k \geq 3$ we can use pairwise independence and Chebyshev. The probability that $|R\cap S|$ deviates from its expectation $E_{k-2}$ by at least $\sigma E_{k-2}$ is at most $$\frac{\V(|R \cap S|)}{(\sigma E_{k-2})^2} \leq \frac{1}{\sigma^2 E_{k-2}} \leq \frac{1}{\sigma^2 q},$$ where we use the fact that $E_{k-2} = |S|/q^{n-k+2} \geq q$ for $k=n-r+3$.
\end{proof}

The next claim gives the `anti concentration' part of the proof overview, showing that a random $(k-2)$-flat in an unbalanced $k$-flat is unbalanced with high probability
\begin{claim}\label{cla-unbalancedflat}
Let $T \in \cL_k(\F_q^n)$ be $\tau$-unbalanced with respect to $S$. Suppose $R$ is chosen  uniformly at random 	from $\cL_{k-2}(T)$. Then
\[ \prob[ R \in B^\sigma_{k-2}(T) ] \geq 1 - \frac{1+\tau}{(\tau-\sigma)^2  q}. \]
\end{claim}
\begin{proof}
As before, the size of $R \cap S$ is a sum of pairwise independent indicator variables with expectation:
\begin{equation}\label{eq-expRcapS}
\E[ |R \cap S| ] = \frac{|S \cap T|}{|T|} q^{k-2} = |S \cap T|/q^2.
\end{equation}

Since $T$ is $\tau$-unbalanced, we have that 
\begin{equation}\label{eq-devScapT}
	| |S \cap T| - E_k | \geq \tau E_k.
\end{equation}
Therefore, dividing by $q^2$ and using (\ref{eq-expRcapS}) we have that
\begin{equation}\label{eq-expRfar}
| \E[ |R \cap S|] - E_{k-2}| \geq \tau E_{k-2}.
\end{equation}

We will separate into two cases: case 1 is when 
\begin{equation}\label{eq-case1}
	\E[|R\cap S|] \leq (1-\tau)E_{k-2}.
\end{equation}
In this case (which can only happen if $\tau < 1$), using Chebyshev, the probability that $R$ is $\sigma$-balanced is bounded from above by,
\begin{eqnarray*}
\prob[ |R\cap S| - E_{k-2} \ge -\sigma E_{k-2}] &\leq& \\
\prob[ |R\cap S| - \E[|R\cap S|] \ge (\tau-\sigma)E_{k-2}] &\leq& \\
	\frac{\V(|R \cap S|)}{(\tau-\sigma)^2 E_{k-2}^2} \leq \frac{\E(|R \cap S|)}{(\tau-\sigma)^2 E_{k-2}^2} &\leq& \frac{1-\tau}{(\tau-\sigma)^2 q}.
\end{eqnarray*}

In the second case we have,
$$\E[ |R \cap S|] \geq (1+\tau)E_{k-2}. $$
In this case the probability that $R$ is $\sigma$-balanced is bounded above by,
\begin{align*}
    \prob[ |R\cap S| - E_{k-2} \leq \sigma E_{k-2}] &\leq \\
    \prob\left[ ||R\cap S| - \E[|R\cap S|]|  \ge \E[|R\cap S|] - (1+\sigma)E_{k-2}\right] &\leq \\
    \prob\left[ ||R\cap S| - \E[|R\cap S|]|  \ge \E[|R\cap S|]\cdot \frac{\tau - \sigma}{1+\tau}\right] &\leq \\
	\frac{\V(|R \cap S|)}{(\tau-\sigma)^2/(1+\tau)^2 \E[|R\cap S|]^2} \leq \frac{(1+\tau)^2}{(\tau-\sigma)^2 \E[|R\cap S|]} \leq \frac{1+\tau}{(\tau-\sigma)^2E_{k-2}} &\leq\frac{1+\tau}{(\tau-\sigma)^2q}.
\end{align*}
Hence, the probability that $R$ is $\sigma$ balanced is bounded by $(1+\tau)/((\tau-\sigma)^2 q)$ and so  we are done.
\end{proof}

We next define three important sets:
\begin{itemize}
    \item {\ ($\cL_{k-2}^*(T)$)} :  For $T\in \cL_k(\F_q^n)$, we define $\cL_{k-2}^*(T)$ to be the set of subspaces in $\cL^*_{k-2}(\F_q^n)$ which on translation can lie in $T$ (or equivalently are parallel to $T$).
    \item { ($\cL_{k-2}(T,\parallel W)$)} :  For $T\in \cL_k(\F_q^n)$ and  $W\in \cL_{k-2}^*(T)$ we let $\cL_{k-2}(T,\parallel W)$ be the set of $(k-2)$ flats in $T$ which are parallel to $W$ (notice that there are exactly $q^2$ such flats and that their disjoint union is $T$).
    \item ($C^{\sigma,c}_{k-2}(T)$) : For $T\in \cL_k(\F_q^n)$  we let $C^{\sigma,c}_{k-2}(T)$ be the set of flats $W$ in $\cL_{k-2}^*(T)$ such that at least a $\left(1-\frac{c(1+\tau)}{(\tau-\sigma)^2q}\right)$-fraction of the flats in $\cL_{k-2}(T,\parallel W)$ are in $B^\sigma_{k-2}(T)$ (we will set $c\ge 1$ to two different values for Theorem~\ref{thm-mainhash} and Theorem~\ref{thm-imphash}). 
\end{itemize}

The previous lemma can now be used to prove that, if $T$ is unbalanced, then many $W$'s are in fact in the set $C^{\sigma,c}_{k-2}(T)$ defined above (this is essentially a Markov style averaging argument).

\begin{claim}\label{clm-manyshifts}
Let $T \in \cL_k(\F_q^n)$ be $\tau$-unbalanced with respect to $S$. Suppose $W$ is chosen  uniformly at random 	from $\cL_{k-2}^*(T)$. Then
\[ \prob[ W \in C^{\sigma,c}_{k-2}(T) ] \geq 1-1/c. \]
\end{claim}
\begin{proof}
    Let us say the claim is false then with probability less than $1-1/c$, $W\in C^{\sigma,c}_{k-2}(T)$ for a uniformly random $W\in \cL_{k-2}^*(T)$. Equivalently, with probability greater than $1/c$, $W\not\in C^{\sigma,c}_{k-2}(T)$. We can sample a uniformly random chosen $R\in \cL_{k-2}(T)$ by first picking a direction $W\in \cL_{k-2}^*(T)$ at random and then taking $R$ to be a random shift of $W$ inside $T$. The above assumption will then give us that:
    \begin{align*}
    \prob[ R \in B^\sigma_{k-2}(T) ]&\le \prob[W\not\in C^{\sigma,c}_{k-2}(T)]\left(1 - \frac{c(1+\tau)}{(\tau-\sigma)^2q}\right)+\prob[W\in C^{\sigma,c}_{k-2}(T)]\\
    &\le 1-\prob[W\not\in C^{\sigma,c}_{k-2}(T)]\frac{c(1+\tau)}{(\tau-\sigma)^2q}<1-\frac{1+\tau}{(\tau-\sigma)^2q}.     
    \end{align*}
    This contradicts Claim \ref{cla-unbalancedflat}.
\end{proof}

Given $W\in \cL^*_{k-2}(\F_q^n)$, let $\cL_{k-2}( \parallel W)$ be the set of $k-2$ flats parallel to $W$ and $\cL^*_k(\parallel W)$ be the set of $k$-dimensional subspaces containing $W$. Let $$B^\sigma_{k-2}(\parallel W)= B^\sigma_{k-2} \cap \cL_{k-2}( \parallel W)$$ denote the set of $\sigma$-unbalanced flats parallel to $W$.

The next claim shows that there is a `good' choice of $W \in \cL_{k-2}^*(\F_q^n)$ to which we should restrict our attention (that is, we will consider only $k-2$ flats parallel to $W$).\footnote{This part of the proof corresponds to the statement in the proof overview arguing that the case of general $k$ and $r$ can be reduced to the case of $r=0$ and $k \mapsto k-r$.} This $W$ should preserve the typical behavior of a random $W$ in two respects: one is that $B^\sigma_{k-2}$ should still have low density when restricted to flats parallel to $W$. The other is that $W$ hits $C^{\sigma,c}_{k-2}(f(A)+A)$ for many $A\in \cL_k^*(\parallel W)$.


\begin{claim}\label{cla-existsW}
	There exists $W \in \cL_{k-2}^*(\F_q^n)$ such that
	\begin{enumerate}
		\item $|B^\sigma_{k-2}(\parallel W)| \leq  \frac{1}{\sigma^2 q(1-\sqrt{1-\delta(1-1/c)})} \cdot |\cL_{k-2}( \parallel W)| \le   \frac{2c}{(c-1)\sigma^2\delta}q^{n-k+1}$
		\item $\prob_{A\sim \cL^*_{k}( \parallel W)}[W\in C^{\sigma,c}_{k-2}(f(A)+A)]\ge 1-\sqrt{1-\delta(1-1/c)}\ge \frac{(c-1)\delta}{c+c\sqrt{1-\delta/2}}\ge \frac{\delta(c-1)}{2c}.$
\end{enumerate}
\end{claim}
\begin{proof}
Let $\alpha=1-\sqrt{1-\delta(1-1/c)}$. Notice, that $B^\sigma_{k-2}$ is a disjoint union of $B^\sigma_{k-2}(\parallel W)$ over all $W \in \cL_{k-2}^*(\F_q^n)$. Suppose $W$ is chosen uniformly at random from $\cL_{k-2}^*(\F_q^n)$ and let $E_1$ be the event $$|B^\sigma_{k-2}(\parallel W)| >  \frac{1}{\sigma^2 q \alpha} \cdot |\cL_{k-2}( \parallel W)|.$$ 
In other words $E_1$ is the event when $W$ does not satisfy 1. above. We then have,
$$\frac{1}{\sigma^2q\alpha}\prob_{W\sim \cL_{k-2}^*(\F_q^n)}\left[E_1\right]\le \prob[ R \in B^\sigma_{k-2} ].$$

Then, by Claim~\ref{cla-Bsmall}, we have that the probability that $W$ does not satisfy 1. above is less than $\alpha=1-\sqrt{1-\delta(1-1/c)}$. 

Consider the bi-partite graph $G$ between $\cL^*_{k}(\F^n_q)$ and $\cL^*_{k-2}(\F^n_q)$ where the edges correspond to pairs $(A,W)\in \cL^*_{k}(\F^n_q)\times \cL^*_{k-2}(\F^n_q)$ such that $W\subset A$. Let $\mu$ be the distribution over the pairs $(A,W)\in \cL^*_{k}(\F^n_q)\times \cL^*_{k-2}(\F^n_q)$ which is uniform over the edges of $G$. As the graph $G$ is regular on both sides sampling from $\mu$ is equivalent to sampling $A$ uniformly from $\cL^*_{k}(\F^n_q)$ and $W$ uniformly from $\cL^*_{k-2}(A)$. It also is equivalent to uniformly sampling $W\in \cL^*_{k-2}(\F^n_q)$ and sampling $A$ from $\cL^*_k(\parallel W)$. By Claim~\ref{clm-manyshifts} and the fact that at least for a $\delta$ fraction of $A\in \cL_k^*(\F^n_q)$, $f(A)+A$ is $\tau$-unbalanced we have,
\begin{equation}\label{eq-midClaim}
\prob_{(A,W)\sim \mu}[W\in C^{\sigma,c}_{k-2}(f(A)+A)]\ge \delta(1-1/c).
\end{equation}
Let $E_2$ be the event 
$$\prob_{A\sim \cL^*_k(\parallel W)}[W\not\in C_{k-2}^{\sigma,c}(f(A)+A)]> \sqrt{1-\delta(1-1/c)}$$ for a random $W$ (that is 2. above is not satisfied). We have,
$$\prob_{W\sim \cL_{k-2}^*(\F_q^n)}[E_1](1-\alpha)\le \prob_{(A,W)\sim \mu}[W\not\in C^{\sigma,c}_{k-2}(f(A)+A)].$$
Using \eqref{eq-midClaim}, we get that the probability that $W$ does not satisfy 2. above is at most $\sqrt{1-\delta(1-1/c)}$. By a union bound we now see that there exists a $W\in \cL^*_{k-2}(\F^n_q)$ which satisfies the two properties in the claim.
\end{proof}

Fix $W = \hat W$ satisfying the two numbered items of Claim~\ref{cla-existsW}. 
Let $G_{\hat W}$ be the random variable which outputs the random 2-flat 
$$f(\spanV\{U,\hat W\})+ U$$ 
for uniformly random $U\in \cL_2^*(\F_q^n)$. Notice that there is a small probability that   $\spanV\{U,\hat W\}$ is not $k$ dimensional. In this case  we set $f(\spanV\{U,\hat W\})=0$.

We will now show that $G_{\hat W}$ has large intersections with a small set with high probability. The particular structure of the random variable $G_{\hat W}$ allows us to state this using the notion of a Furstenberg set.

\begin{claim}\label{cla-Kakeyaset}
There exists a set $K \subset \F_q^n$ such that
\begin{enumerate}
	\item $|K| \leq \frac{2c}{(c-1)\sigma^2\delta}q^{n-1}.$
	\item $K$ is $\left(2,\left(1-\frac{c(1+\tau)}{(\tau-\sigma)^2q}\right)q^2,\delta\frac{c-1}{2c}(1-1/q-1/q^2)\right)$-Furstenberg.

\end{enumerate}
\end{claim}
\begin{proof}
	We take
	\[ K = \bigcup_{R \in B^\sigma_{k-2}(\parallel \hat W)} R,\]
	to be the union of all $\sigma$-unbalanced $(k-2)$-flats parallel to $\hat W$. To show that 1. holds, we use the first item of Claim~\ref{cla-existsW} and the fact that each $R$ has $q^{k-2}$ points.

    For a uniformly random $U\in \cL^*_2(\F^n_q)$, $F=f(\spanV(U,\hat W))+U$ gives us a sample from $G_{\hat W}$. Let $T=f(\spanV(U,\hat W))+\spanV(U,\hat W)$. If $\hat W \in C^{\sigma,c}_{k-2}(T)$ then that means at least
    $$\left(1-\frac{c(1+\tau)}{(\tau-\sigma)^2q}\right)q^2$$ 
    many flats in $\cL_{k-2}(T,\parallel \hat W)$ are in $B^\sigma_{k-2}(T)$ and hence contained in $K$. $F$ will intersect with each of these flats (and hence $K$) in distinct points. This is because $T$ is a $f(\spanV(U,\hat W))$-shift of the span of $U$ and $\hat W$ and $U\cap \hat W=\{0\}$ so $T$ is a disjoint union of shifts of $\hat W$ by elements in $F=f(\spanV(U,\hat W))+U$. This implies that $F$ is $(1-c(1+\tau)/(\tau-\sigma)^2q)q^2$-rich with respect to $K$. Finally, note that conditioned on the event that $\spanV(U,\hat W)$ is $k$-dimensional $T$ has the same distribution as $f(A)+A$ where $A$ is uniformly distributed over $\cL^*_k(\parallel \hat W)$. This means
	\begin{align*}
	    \prob\left[ G_{\hat W} \text{ is } \left(1-\frac{c(1+\tau)}{(\tau-\sigma)^2q}\right)q^2\text{-rich with respect to } K\right]\ge \\ \prob_{A\sim\cL_k^*(\parallel \hat W)}\left[{\hat W}\in C^{\sigma,c}_{k-2}(f(A)+A)\right]\cdot \prob_{U\in \cL^*_2(\F_q^n)}\left[\dim\spanV\{U,\hat W\}=k\right].
	\end{align*}
	
	We note $\prob_{U\in \cL^*_2(\F_q^n)}[\dim\spanV\{U,\hat W\}=k]$ is at least $1-1/q-1/q^2$. If we generate $U$ by picking two random vectors then the first one being in $\hat{W}$ has probability at most $1/q^{n-k+2}\le 1/q^2$ and the second being in the space spanned by the first and $\hat{W}$ has probability at most $1/q^{n-k+1}\le 1/q$. Now using Claim \ref{clm-manyshifts} and the equation above we have,
	$$\prob\left[ G_{\hat W} \text{ is } \left(1-\frac{c(1+\tau)}{(\tau-\sigma)^2q}\right)q^2\text{-rich}\right]\ge \frac{(c-1)\delta}{2c}\left(1-\frac{1}{q}-\frac{1}{q^2}\right).$$
	
	The above equation implies 2. as $G_{\hat W}$ by definition takes a uniformly chosen $U\in \cL^*_2(\F_q^n)$ and outputs a flat parallel to $U$.
\end{proof}

To finish the proof of the theorem we need a bound for $(2,\gamma q^2,\beta)$-Furstenberg Sets. We will use Lemma~\ref{lem-fur2-intro} which we prove in the next section. We restate the Lemma here for convenience.

\RestateGo{\restatefur}
\lemfurii*
Given this lemma, we substitute the values 
$$\gamma = \left(1-\frac{c(1+\tau)}{(\tau-\sigma)^2q}\right),\beta=\frac{\delta(c-1)}{2c}\left(1-\frac{1}{q}-\frac{1}{q^2}\right) $$ and the bound on $|K|$ given by Claim~\ref{cla-Kakeyaset} into the lemma above. We get the bound,
\begin{equation}\label{eq-bifur}
    \frac{2c}{(c-1)\sigma^2\delta}q^{n-1}\ge |K|\ge q^n \left(1-\frac{c(1+\tau)}{(\tau-\sigma)^2q}\right)^n \left(1+\frac{1}{q}\right)^{-n}\frac{\delta(c-1)}{2c}\left(1-\frac{1}{q}-\frac{1}{q^2}\right).
\end{equation}

To prove Theorem~\ref{thm-mainhash}  use $q\ge 32\max(n(1+\tau)/(\tau\delta)^2,n)$, $c=4$, $\sigma=\tau/2$ and $\delta\le 1$ in \eqref{eq-bifur} and re-arrange to get:
$$\frac{8}{9n} \ge \left(1-\frac{1}{2n}\right)^n\left(1+\frac{1}{32n}\right)^{-n}\left(1-\frac{1}{32n}-\frac{1}{32^2n^2}\right).$$
Using $(1-x/n)^n\ge e^{-x}(1-x^2/n)$ for $x<n$, $(1+x/n)^n\le e^x$ and $n\ge 5$ then implies:
$$\frac{8}{45}> e^{-1/5}(1-1/20)e^{-1/32}(1-1/160-1/(160)^2)$$
which leads to a contradiction proving Theorem~\ref{thm-mainhash}.

To prove Theorem~\ref{thm-imphash} use $q\ge\max(n(1+\tau)/(\tau-\sqrt{\tau})^2\delta^2,n)$,$\sigma=\sqrt{\tau}$, $\delta\le 1/10$ and set $c=10$ in \eqref{eq-bifur} to get:
$$\frac{400}{81n}\ge\frac{400(\tau-\sqrt{\tau})^2}{81\tau(\tau+1)n}\ge\left(1- \frac{1}{10n}\right)^n\left(1+\frac{1}{n}\right)^{-n}\left(1-\frac{1}{n}-\frac{1}{n^2}\right).$$

Using $(1-x/n)^n\ge e^{-x}(1-x^2/n)$ for $x<n$, $(1+x/n)^n\le e^x$ and $n\ge 20$ gives us:
$$\frac{400}{81\cdot 20}>e^{-1/10}(1-1/(100\cdot 20))e^{-1}(1-/20-1/400)$$
which leads to a contradiction proving Theorem~\ref{thm-imphash}.
\end{proof}

\paragraph{Modifications to prove Theorem~\ref{thm-smoothSlice}:}
In the setting of Theorem~\ref{thm-smoothSlice} we have $E_{k-2}=|S|/q^{k-2}\ge q^\eta$. We see the statements of the various claim can be appropriately modified to prove Theorem~\ref{thm-smoothSlice}. We state the appropriate modification of the main claims proven assuming Theorem~\ref{thm-smoothSlice} is false. That is there exists a function $f: \cL^*_k(\F_q^n)\rightarrow \F_q^n$ (with parameters as in Theorem~\ref{thm-smoothSlice}) such that for a $\delta$ fraction of $A\in \cL_k^*(\F_q^n)$, the flat $f(A)+A$ is $\tau$ unbalanced with respect to $S$. We do not give the proofs as the arguments are identical.

\begin{claim}
If $R$ is chosen uniformly in $\cL_{k-2}(\F_q^n)$ then 
\[ \prob[ R \in B^\sigma_{k-2} ] \leq \frac{1}{\sigma^2 q^\eta}. \]
\end{claim}

\begin{claim}
Let $T \in \cL_k(\F_q^n)$ be $\tau$-unbalanced with respect to $S$. Suppose $R$ is chosen  uniformly at random 	from $\cL_{k-2}(T)$. Then
\[ \prob[ R \in B^\sigma_{k-2}(T) ] \geq 1 - \frac{1+\tau}{(\tau-\sigma)^2  q^\eta}. \]
\end{claim}

In this proof we redefine $C^{\sigma,c}_{k-2}(T)$ as follows: For $T\in \cL_k(\F_q^n)$  we let $C^{\sigma,c}_{k-2}(T)$ be the set of flats $W$ in $\cL_{k-2}^*(T)$ such that at least a $\left(1-\frac{c(1+\tau)}{(\tau-\sigma)^2q^\eta}\right)$-fraction of the flats in $\cL_{k-2}(T,\parallel W)$ are in $B^\sigma_{k-2}(T)$. 

\begin{claim}
Let $T \in \cL_k(\F_q^n)$ be $\tau$-unbalanced with respect to $S$. Suppose $W$ is chosen  uniformly at random 	from $\cL_{k-2}^*(T)$. Then
\[ \prob[ W \in C^{\sigma,c}_{k-2}(T) ] \geq 1-1/c. \]
\end{claim}

\begin{claim}
	There exists $W \in \cL_{k-2}^*(\F_q^n)$ such that
	\begin{enumerate}
		\item $|B^\sigma_{k-2}(\parallel W)| \leq  \frac{1}{\sigma^2 q^\eta(1-\sqrt{1-\delta(1-1/c)})} \cdot |\cL_{k-2}( \parallel W)| \le   \frac{2c}{(c-1)\sigma^2\delta}q^{n-k+2-\eta}$
		\item $\prob_{A\sim \cL^*_{k}( \parallel W)}[W\in C^{\sigma,c}_{k-2}(f(A)+A)]\ge 1-\sqrt{1-\delta(1-1/c)}\ge \frac{(c-1)\delta}{c+c\sqrt{1-\delta/2}}\ge \frac{\delta(c-1)}{2c}.$
\end{enumerate}
\end{claim}

Using the previous claims we can prove the next claim that will contradict Lemma~\ref{lem-fur2-intro} completing the proof.

\begin{claim}[Furstenberg Set construction from assuming Theorem~\ref{thm-smoothSlice} is false]
There exists a set $K \subset \F_q^n$ such that
\begin{enumerate}
	\item $|K| \leq \frac{2c}{(c-1)\sigma^2\delta}q^{n-\eta}.$
	\item $K$ is $\left(2,\left(1-\frac{c(1+\tau)}{(\tau-\sigma)^2q^\eta}\right)q^2,\delta\frac{c-1}{2c}(1-1/q-1/q^2)\right)$-Furstenberg.
\end{enumerate}
\end{claim}

\subsection{The case of \texorpdfstring{$\F_2$}{F2}}\label{sec-binary}

In this section we prove Theorem~\ref{thm-binary} using Theorem~\ref{thm-mainhash}. The same argument can be used to derive Theorem~\ref{thm-impbinary} from Theorem~\ref{thm-imphash}. We restate the theorem for convenience. 

\RestateGo{\restatebinary}
\theobinary*
\begin{proof}
	Take $$ \ell = \lceil \log_2(32\max(n(1+\tau)/(\tau\delta)^2,n)) \rceil.$$ and set $$q = 2^\ell.$$ Let $$n' = \lceil n/\ell \rceil $$ so that  we have $$2^n \leq q^{n'}.$$ We now identify $\F_2^n$ with an $\F_2$-linear subspace of $\F_q^{n'}$, e.g., by identifying $\F_q^{n'}$ with $\F_2^{n'\ell}$ as $\F_2$-vector spaces and then identifying $\F_2^n$ with the first $n \leq n'\ell$ coordinates (the rest can be set to zero). The above embedding of $\F_2^n$ in $\F_q^{n'}$ allows us to think of the set $S$ as sitting in $\F_q^{n'}$ and so we can apply Theorem~\ref{thm-mainhash} if we check that all the conditions are met. We first see that, by our choice of $\ell$, the bound on $q \geq 32\max(n'(1+\tau)/(\tau\delta)^2,n')$ is met (notice that $n' \leq n$). We also need to  check that $|S| > q^4$ which holds from our assumption $|S| \geq 2^{20}\max(n^4(1+\tau)^4/(\tau\delta)^8,n^4)$. $n'\ge 5$ is also satisfied.
	
	Hence we can apply Theorem~\ref{thm-mainhash} in our setting. Let $r$ be such that $$q^r < |S| \leq q^{r+1}$$ and set $$t'= r-3.$$
	
	We get that for a $(1-\delta)$-fraction of all surjective linear maps $L' : \F_q^{n'} \to \F_q^{t'}$ satisfy the property that $L'(U_S)$ is $\tau q^{-t'}$-close to uniform in the $\ell_\infty$ distance. Since an $\F_q$-linear map is also an $\F_2$-linear map, we can think of $L'$ as an $\F_2$-linear map from $\F_2^{n'\ell}$ to $\F_2^{t' \ell}$. Setting $$t = t'\ell$$ and let $L$ be the restriction of $L'$ to the subspace we previously identified with $\F_2^n$ (which contains~$S$) we get that for any such $L : \F_2^n \to \F_2^t$, $L(U_S)$ is $\tau 2^{-t}$-close to uniform in the $\ell_\infty$ distance (clearly $L(U_S)$ and $L'(U_S)$ have the same distribution).
	
	We now bound the `entropy loss' or $ \log_2|S| - t.$  Notice that $$\log_2|S| \leq (r+1)\ell$$ and that $$t =  t'\ell = (r-3)\ell.$$ Combining the last two inequalities we get that 
	$$ \log_2|S| -  t \leq 4\ell \le 4 \log_2(\max(n(1+\tau)/(\tau\delta)^2,n)-20.$$ 
	
	We are not done yet as not all surjective linear maps from $\F_2^n$ to $\F_2^t$ will be restrictions of surjective linear maps from $\F_q^{n'}$ to $\F_q^{t'}$. We now overcome this obstacle using a random rotation argument. We started out with embedding $S\subseteq \F_2^n$ in a bigger space $\F_2^{n'\ell}$. If we can show a $(1-\delta)$-fraction of surjective linear maps from $\F_2^{n'\ell}$ to $\F_2^{t}$ satisfy the desired property we are also done. We also let $\phi:\F_2^{n'\ell}\rightarrow \F_q^{n'}$ be the $\F_2$-linear isomorphism between $\F_2^{n'\ell}$ and $\F_q^{n'}$ we had implicitly chosen in the beginning.
	
	Let $H$ be the set of surjective linear maps from $\F_2^{n'\ell}$ to $\F_2^t$ which are also surjective linear maps from $\F_q^{n'}$ to $\F_q^{t'}$ (indeed every surjective linear map from $\F_q^{n'}$ to $\F_q^{t'}$ is a surjective linear map from $\F_2^{n'\ell}$ to $\F_2^t$ but the converse is not the case). We just showed that a $(1-\delta)$-fraction of the maps in $H$ satisfy the desired property. Let $M$ be a random invertible linear map in $\text{GL}_{n'\ell}(\F_2)$. We note $\phi\circ M$ is also a valid $\F_2$-linear isomorphism between $\F_2^n{n'\ell}$ and $\F_q^{n'}$. If we repeated our earlier argument with this isomorphism we will have proven that a $(1-\delta)$-fraction of the maps in $M\cdot H=\{L\circ M| L\in H\}$ satisfy the desired property. But under a random rotation we see that each surjective linear map from $\F_2^{n'\ell}$ to $\F_2^t$ will be included in an equal number of $M\cdot H$. This proves that there is at least a $(1-\delta)$-fraction of surjective linear maps from $\F_2^{n'\ell}$ to $\F_2^t$ which satisfy the desired property.
\end{proof}

\section{Proof of Lemma~\ref{lem-fur2-intro} using the polynomial method}\label{sec-furstenberg-proof}

We will be using the polynomial method to lower bound the sizes of $(2,\gamma q^2,\beta)$-Furstenberg sets in $\F_q^n$ which are needed to prove our hashing guarantees. As stated earlier, these bounds have been proven in \cite{ORW22} using a combinatorial reduction. The bounds from \cite{ORW22} can be directly used to prove our hashing theorems with slightly worse constants. 

We will give a new proof to lower bound these set sizes by extending ideas developed in~\cite{dhar2021kakeya} to prove bounds for Kakeya sets over rings of integers modulo a composite number. The advantages are three fold: we get slightly better constants, the argument here gives significantly better bounds for $(1,\gamma q,\beta)$-Furstenberg sets (although not important for our application) and as mentioned earlier these ideas were later used to resolve the maximal Kakeya conjecture over rings of integers modulo a composite number~\cite{dhar2022maximal}.

In this section we develop improvements to the polynomial method argument to get the desired dependence on $\beta$. In a nutshell, our improvement comes from picking a carefully chosen set of monomials, instead of just taking all monomials up to a specified degree. This section will be divided into three sub-sections. First, we review basic definitions and results on the polynomial method (with multiplicities) as developed in \cite{DKSS13}. Then, we devote a section to understanding ranks of sub-matrices of a special  matrix which maps a polynomial to its evaluations (with derivatives) on a given set of points. Finally, we put everything together to prove  Lemma~\ref{lem-fur2-intro}.

\subsection{Multiplicities and Hasse derivative}
 We first review the definitions of multiplicities and Hasse derivatives that will be needed in the proof (see  \cite{DKSS13} for a more detailed discussion). We will allow the definitions to be over an arbitrary field $\F$ since we will need to apply them both for $\F=\F_q$ (which is the usual case) and also for  $\F = \F_q(t_1,t_2)$ (the field of rational function in $t_1,t_2$ with coefficients in $\F_q$). Working over this extension field is natural when handling two-dimensional flats and already appears in \cite{KLSS2011}.

\begin{definition}[Hasse Derivatives]
Let $\F$ be a field. Given a polynomial $f\in \F[x_1,\hdots,x_n]$  and an $\mathbf{i}\in \Z_{\ge 0}^n$ the $\mathbf{i}$th {\em Hasse derivative} of $f$ is the polynomial $f^{(\mathbf{i})}$ in the expansion $$f(x+z)=\sum_{\mathbf{j}\in \Z_{\ge 0}^n} f^{(\mathbf{j})}(x)z^{\mathbf{j}}$$ where $x=(x_1,...,x_n)$, $z=(z_1,...,z_n)$ and $z^{\mathbf{j}}=\prod_{k=1}^n z_k^{j_k}$.  
\end{definition}

Hasse derivatives satisfy the following useful property (see \cite{DKSS13} for a proof). We will only need this property  to show that, if $f^{(\mathbf{i}+\mathbf{j})}$ vanishes at a point then so does $(f^{(\mathbf{i})})^{(\mathbf{j})}$.

\begin{lemma}\label{lem:chainRule}

Given a polynomial $f\in \F[x_1,\hdots,x_n]$ and $\mathbf{i},\mathbf{j}\in \Z_{\ge 0}^n$, we have 
$$(f^{(\mathbf{i})})^{(\mathbf{j})}=f^{(\mathbf{i}+\mathbf{j})}\prod\limits_{k=1}^n\binom{i_k+j_k}{i_k}$$
\end{lemma}

We make precise what it means for a polynomial to vanish on a point $a\in \F^n$ with multiplicity. First we recall for a point $\mathbf{j}$ in the non-negative lattice $\Z^n_{\ge 0}$, its weight is defined as $\text{wt}(\mathbf{j})=\sum_{i=1}^n j_i$.

\begin{definition}[Multiplicity]
For a polynomial $f\in \F[x_1,\hdots,x_n]$ and a point $a\in \F^n$ we say $f$ vanishes on $a$ with {\em multiplicity} $m\in \Z_{\geq 0}$, if $m$ is the largest integer such that all Hasse derivatives of $f$ of weight strictly less than~$m$ vanish on $a$. We use {\em $\textsf{mult}(f,a)$} to refer to the multiplicity of $f$ at $a$.
\end{definition}

Note that the number of Hasse derivatives over $\F[x_1,\hdots,x_n]$ with weight strictly less than~$m$ is $\binom{n+m-1}{n}$. Hence, requiring that a polynomial vanishes to order $m$ at a single point $a$ enforces the same number of homogeneous linear equations on the coefficients of the polynomial.  We will use the following simple property concerning multiplicities of composition of polynomials (see \cite{DKSS13} for a proof).

\begin{lemma}\label{lem:multComp}
Given a polynomial $f\in \F[x_1,\hdots,x_n]$ and a tuple $H=(h_1,\hdots,h_n)$ of polynomials in $\F[y_1,\hdots,y_m]$, and $a\in \F^m$ we have, 
{\em $$\textsf{mult}(f\circ H, a)\ge \textsf{mult}(f,H(a)).$$}
\end{lemma}

We will now state the multiplicity version of the  Schwartz-Zippel bound~\cite{schwartz1979probabilistic,ZippelPaper} 
(see \cite{DKSS13} for a proof). We denote by $\F[x_1,..,x_n]_{\le d}$ the space of polynomials of total degree at most $d$ with coefficients in $\F$.

\begin{lemma}[Schwartz-Zippel with multiplicities]\label{multSchwartz}
Let $\F$ be a field, $d\in \Z_{\geq 0}$ and let $f\in \F[x_1,..,x_n]_{\le d}$ be a non-zero polynomial. Then, for any finite subset $U\subseteq \F$ ,
{\em $$\sum\limits_{a\in U^n} \textsf{mult}(f,a) \le d|U|^{n-1}.$$}
\end{lemma}

\subsection{The EVAL matrix, its submatrices and their ranks}

If $M$ is a matrix over an extension field of $\F_q$, we define the $\F_q$-rank of $M$, denoted by $\rank_{\F_q} M$, to be the size of the largest subset of columns of $M$ which are $\F_q$-linearly independent (in other words, no non-zero $\F_q$-linear combination of those columns is $0$). For convenience, we define the coefficient matrix of a matrix with entries in $\F_q[t_1,t_2]$. This will help us argue about the $\F_q$-rank of a matrix over an extension, by connecting it with the rank of a matrix with entries in~$\F_q$.

\begin{definition}[Coefficient matrix of $E$]
    Let $E$ be  an $n_1\times n_2$ matrix with entries in $\F_q[t_1,t_2]_{\leq d}$.  The {\em coefficient matrix of $E$}, denoted by {\em $\Coeff(E)$}, is a $\binom{d+2}{2}n_1\times n_2$ matrix with entries in $\F_q$ whose rows are labelled by elements in $((i,j),k)\in\Z_{\ge 0}^2\times [n_1]$ and whose entry in row $((i,j),k)$ and column $\ell$ is given by the coefficient of $t_1^it_2^j$ of the polynomial in the $(k,\ell$)'th entry of $E$. 
\end{definition}
In other words, to construct $\Coeff(E)$ we replace each entry with a (column) vector of its coefficients.  For example: 
$$
 E=\begin{bmatrix}
 t_1 & t_2 + 1 \\
 2+4t_1 & t_1+3t_2
 \end{bmatrix}
 ,\ \ \Coeff(E)=\begin{blockarray}{*{2}{c} l}
\begin{block}{[*{2}{c}] l}
  0 & 1 \bigstrut[t]\\
  1 & 0 \\
  0 & 1 \\
  2 & 0 \\
  4 & 1 \\
  0 & 3 \\
\end{block}
\end{blockarray}.
$$

By construction we have,
    $$\rank_{\F_q} E=\rank_{\F_q} \Coeff(E).$$

Our main object of interest is the matrix encoding the evaluation of a subset of monomials (with their derivatives) on a subset of points.

\begin{definition}[$\EVAL^m(S,W)$ matrix]
Let $\F$ be a field, and let $n,m \in \mathbb{N}$. Given a set $S \subset \F^n$ and a set of monomials $W\subset \F[x_1,\hdots,x_n]$, we let {\em $\EVAL^m(S,W)$} denote an  $|S|\binom{m-1+n}{n}\times |W|$ matrix whose columns are indexed by $W$ and rows are indexed by tuples $(x,\mathbf{j})\in S\times \Z_{\ge 0}^n$ such that $\text{wt}(\mathbf{j})<m$.
The $((x,\mathbf{j}),f)$th entry of this matrix is,
$$f^{(\mathbf{j})}(x).$$
In other words, the $(x,\mathbf{j})$th row of the matrix consists of the evaluation of the $\mathbf{j}$'th Hasse derivative of all $f\in W$ at $x$. Equivalently, the $f$'th column of the matrix consists of the evaluations of weight strictly less than $m$ Hasse derivatives of $f$ at all points in $S$.
\end{definition}

We let,
$$\mathcal{V}=\{u't_1+v't_2|u',v'\in \F_q\}^n=\{ut_1+vt_2|u,v\in \F_q^n\}\subseteq (\F_q[t_1,t_2])^n$$ 
denote the set of $n$-tuples of homogeneous linear forms in $t_1,t_2$
and 
$$\mathcal{V}_{\text{full}}=\{ut_1+vt_2\in \cV\,|\, \dim_{\F_q} \spanV\{u,v\}=2\}\subseteq \mathcal{V}$$ denote the subset of $\cV$ in which the coefficient vectors of $t_1$ and of $t_2$ are linearly independent.

Let $W_{d,n}$ denote the set of monomials in $n$-variables $x_1,\hdots,x_n$ of degree at most $d$.  Our first lemma shows that the $\F_q$-rank of $\EVAL^m(\cV,W_{d,n})$ is maximal whenever $d$ is not too large. This is essentially the Schwartz-Zippel lemma since it means that a polynomial of bounded degree could be recovered from its evaluations (up to high enough order) on a product set.

\begin{lemma}[Rank of $\EVAL^m(\cV,W_{d,n})$]
Let $m\in \N$ then for all $d<mq^2$ we have,
{\em $$\text{rank}_{\F_q} \EVAL^m(\cV,M_{d,n})=|W_{d,n}|= \binom{d+n}{d}.$$}
\end{lemma}
\begin{proof}
    Recall $\cV=\{u't_1+v't_2 \in \F_q(t_1,t_2)|u,v\in \F_q\}^n$. 
    Any $\F_q$-linear combination of columns in $\EVAL^m(S,W_{d,n})$ for some subset $S\subseteq \cV$ corresponds to looking at the evaluation of the weight  $< m$ Hasse derivatives on $S$ of a degree at most $d$ polynomial in $\F_q[x_1,\hdots,x_n]$. To be precise if we take the linear combination of columns corresponding to the monomials $f_1,f_2,\hdots,f_\ell$ with coefficients $\alpha_1,\hdots,\alpha_\ell\in \F_q$ the column vector we get will be the evaluation of the weight  $<m$ Hasse derivatives of $\sum_{i=1}^\ell \alpha_if_i\in \F_q[x_1,\hdots,x_n]$ over $S$. Note that the polynomial we are considering has coefficients only in $\F_q$ while the evaluations are being done over the field $\F_q(t_1,t_2)$.
    
    For any $d<mq^2$ any $\F_q$-linear combination of columns in $\EVAL^m(\cV,W_{d,n})$ being $0$ will be equivalent to a degree at most $mq^2-1$ polynomial vanishing on $\cV$ with multiplicity $m$. By Lemma \ref{multSchwartz} we see that a non-zero polynomial of degree at most $d$ vanishing on 
    $\cV$ (which is a product set of size $q^{2n}$)
    with multiplicity at least $m$ satisfies 
    $$dq^{2(n-1)}\ge mq^{2n}$$ 
    which leads to a contradiction (as $d<mq^2$). This means $\EVAL^m(\cV,W_{d,n})$ has $\F_q$-rank $|W_{d,n}|$ for $d<mq^2$. Note this proof would also show that the $\F_q(t_1,t_2)$-rank of $\EVAL^m(\cV,W_{d,n})$ is $|W_{d,n}|$ for $d<mq^2$.
\end{proof}
We next show that the same rank bound holds even if we restrict the rows to only come from the smaller set $\cV_{\text{full}}$.
 \begin{lemma}
    {\em $\EVAL^m(\cV_{\text{full}},W_{d,n})$ has $\F_q$-rank $|W_{d,n}|$ for $d<mq^2$.}
\end{lemma} 
    \begin{proof}
        This lemma will need the fact that we are only computing the $\F_q$ (and not $\F_q(t_1,t_2)$) rank. Consider any $\F_q$-linear combination $f$ of monomials in $W_{d,n}$. It suffices to show that if~$f$ vanishes with multiplicity at least $m$ over $\cV_{\text{full}}$ then it vanishes with multiplicity at least $m$ over~$\cV$. $\cV\setminus \cV_{\text{full}}$ contains elements of the form $ut_1$ or $u(ct_1+t_2)$ where $u\in \F_q^n$ and $c\in \F_q$. First we consider $u\in \F_q^n\setminus \{0\}$. We can pick a $v\in \F_q^n$ such that $v$ and $u$ are linearly independent. $ut_1+vt_2$ now is an element in $\cV_{\text{full}}$. This means~$f$ vanishes on $ut_1+vt_2$ with multiplicity at least $m$. $f$ is a polynomial in $\F_q[x_1,\hdots,x_n]$ which means all its Hasse derivatives are also $\F_q$-polynomials. Therefore, for any $\mathbf{i}\in \Z_{\ge 0}^n$ we get $f^{(\mathbf{i})}(ut_1)$ by setting $t_2=0$ in $f^{(\mathbf{i})}(ut_1+vt_2)$. This implies that $f$ vanishes on $ut_1$ with multiplicity at least $m$. Setting $t_1=0$ then shows that $f$ vanishes on $0$ with multiplicity at least $m$. Again as $t_1$ is a formal variable and $f\in \F_q[x_1,\hdots,x_n]$ we can replace $t_1$ with $ct_1+t_2$ to get $f$ vanishes on $u(ct_1+t_2)$ with multiplicity at least $m$.  
\end{proof} 
    
Our final lemma, which is the heart  of this section, shows that any $\delta$-fraction of the rows in $\EVAL^m(\cV_{\text{full}},W_{d,n})$ have rank at least $\delta$ times the rank of the full matrix. This is not true for an arbitrary matrix and uses the fact that the general linear group acts on the set of rows in a transitive way.

\begin{lemma}[Rank of $\EVAL^m(S,W_{d,n})$]
Let $m\in \N$ and $S\subseteq \mathcal{V}_{\text{full}}$ with $|S|\ge \delta |\mathcal{V}_{\text{full}}|,\delta\in [0,1]$ then for all $d<mq^2$ we have,
{\em $$\text{rank}_{\F_q} \EVAL^m(S,W_{d,n})\ge \delta  \cdot|W_{d,n}|=\delta \binom{d+n}{d}.$$}
\end{lemma}
\begin{proof}
    Consider $S\subseteq \cV_{\text{full}}$ such that $|S|=\delta |\cV_{\text{full}}|$. For any $M\in \text{GL}_n(\F_q)$ we let $M$ act on $ut_1+vt_2$ where $u,v\in \F_q^n$ as $M\cdot(ut_1+vt_2)=Mut_1+Mvt_2$. Let $M\cdot S=\{M\cdot y|y\in S\}$.
    \begin{claim}
    {\em $$\rank_{\F_q}\EVAL^m(S,W_{d,n})=\rank_{\F_q}\EVAL^m(M\cdot S,W_{d,n}).$$}
    \end{claim}
    \begin{subproof}
        We will prove this statement by constructing an isomorphism between the column-space of the two matrices. An element in the column space of $\EVAL^m(S,W_{d,n})$ is the evaluation of the weight strictly less than $m$ Hasse derivatives on $S$ of a polynomial $f(x)\in \F_q[x_1,\hdots,x_n]$ of degree at most~$d$. We map such a vector to the evaluation of the weight strictly less than $m$ Hasse derivatives on $M\cdot S$ of the polynomial $f(M^{-1}x)$ which will also be of degree at most~$d$. The choice of $f$ in the beginning can be ambiguous but if there are two polynomials $f(x)$ and $g(x)$ having the same evaluation of weight strictly less than $m$ Hasse derivatives over $S$ then $f(x)-g(x)$ vanishes on $S$ with multiplicity at least $m$. By Lemma~\ref{lem:chainRule}, $f(M^{-1}x)-g(M^{-1}x)$ vanishes on $M\cdot S$ with multiplicity $m$ which implies $f(M^{-1}x)$ and $g(M^{-1}x)$ evaluate to the same weight strictly less than $m$ Hasse derivatives over $M\cdot S$.
        The inverse map can be similarly constructed.
    \end{subproof}
    The above claim shows it suffices to show the rank bound for any $M\cdot S$ where $M\in \text{GL}_n(\F_q)$. We do this by a probabilistic method argument.
    
    The previous Lemma implies that $\Coeff(\EVAL^m(\cV_{\text{full}},W_{d,n}))$ has $\F_q$-rank $|W_{d,n}|$. As this is a matrix with $\F_q$ entries this means that there exists a $|W_{d,n}|=\binom{d+n}{n}$ subset of rows $R$ of $\Coeff(\EVAL^m(\cV_{\text{full}},W_{d,n}))$ which are linearly independent.
    These rows are indexed by tuples 
    $$(x,\bi,(j,k))\in \cV_{\text{full}}\times \Z_{\ge 0}^n\times \Z_{\ge 0}^2$$
    with $\text{wt}(\bi)<m$ and $j+k\le d$. The $(x,\bi,(j,k))$th row is the coefficient of $t_1^jt_2^k$ in the evaluation of the $\bi$th Hasse Derivative at $x$ of the monomials in $W_{d,n}$.

    We pick an $M\in \text{GL}_n(\F_q)$ uniformly at random. We now calculate the expected fraction of the rows from $R$ which appear in $\Coeff(\EVAL^m(M\cdot S, M_{d,n}))$. 
    
    A row in $R$ indexed by $(x,i,(j,k))\in \cV_{\text{full}}\times \Z_{\ge 0}^n\times \Z_{\ge 0}^2$ will appear in $\Coeff(\EVAL^m(M\cdot S, W_{d,n}))$ if and only if $x\in M\cdot S$. As the action of $\text{GL}_n(\F_q)$ on $\cV_{\text{full}}$ we see that this happens with probability at least $\delta$. This means the expected fraction of rows in $R$ appearing in $\Coeff(\EVAL^m(M\cdot S, W_{d,n}))$ is at least $\delta$. This ensures that there is some matrix $M$ such that $\Coeff(\EVAL^m(M\cdot S, W_{d,n}))$ and hence $\EVAL^m(M\cdot S, W_{d,n})$ has $\F_q$-rank at least $\delta|W_{d,n}|$. 
\end{proof}
We note the above lemma could be proven in a more general setting where we wanted to compare the $\F_q$ rank of $\EVAL^m(G,W_{d,n})$ for $G=S\subseteq \F^n$ and $G=S'\subseteq S$ a large subset of $S$ as long as the general linear group acts transitively on $S$. For instance, this style of argument was also used in \cite{dhar2021kakeya} to obtain a better dependence on $\beta$ for $(1,m,\beta)$-Furstenberg sets\footnote{Denoted as $(m,\beta)$-Kakeya sets in \cite{dhar2021kakeya}.} over $\Z/p^k\Z$. 

We will use a simple corollary of this lemma.
\begin{corollary}\label{cor:goodMono}
Let $r\in \N$ and $S\subseteq \cV_{\text{full}}$ with $|S|\ge \delta |\cV_{\text{full}}|,\delta\in [0,1]$ then for any $d<rq^2$ there exists a set $P_S(d,r), |P_S(d,r)|=\delta \binom{d+n}{n}$ of monomials of degree at most $d$ such that no non-zero $\F_q$-linear combination of monomials in $P_S(d,r)$ vanishes with multiplicity at least $r$ over all points in $S$. 
\end{corollary}

\subsection{Proving the bound on Furstenberg sets}

We first give a brief description of the  polynomial method argument as was used for example in~\cite{KLSS2011}. Given a $(k,\gamma q^2,\beta)$-Furstenberg set $K$ we take a polynomial $Q$ of degree at most $d$ (where~$d$ depends on $\beta, \gamma$ and $q$) which vanishes with high multiplicity on $K$. If $|K|$ small, such a polynomial can be found by solving a system of linear constraints. For at least a $\beta$ fraction of the flats $A\in \cL_2^*(\F^n_q)$ there is a shift $a+A,a\in \F_q^n$ such that $a+A$ is $\gamma q^2$-rich with respect to $K$. By restricting $Q$ to $a+A$ and using Lemma~\ref{multSchwartz} we then show that $Q$ vanishes identically on $a+A$ which implies that the highest degree homogeneous part of $Q$ vanishes identically  on $A$. This will imply that the highest degree homogenous part of $Q$ vanishes on a $\beta$ fraction of $A\in \cL^*_2(\F^n_q)$. Another application of the Lemma~\ref{multSchwartz} then gives us a size bound for $|K|$ by arguing that $\deg(Q)$ cannot be too small (here, the dependency between $\beta$ and $d$ comes into play). The size of $K$ is lower bounded by the number of at most degree $d$ monomials $\binom{d+n}{n}$. The dependence of $\beta$ on $d$ leads to  a loss of $\beta^n$ in the final bound.

In \cite{ORW22} they overcome this problem by using random rotations to reduce to the case of constant $\beta$. We overcome this loss by instead using Corollary~\ref{cor:goodMono} to start out with a subset of monomials of degree at most $d'$ (here $d'$ will not depend on $\beta$) of size $\beta \binom{d'+n}{n}$ such that any $\F_q$-linear combination of those will not vanish on the $\beta$ fraction of flats in $\cL^*_k(\F^n_q)$ which have $\gamma q^2$-rich shifts with respect to $K$. Now the standard polynomial method argument will let us prove Lemma~\ref{lem-fur2-intro}.

\lemfurii*


\begin{proof}
Let 
$$t=\lceil \gamma q^2\rceil/q\ge \gamma q.$$
As $K$ is a $(2,\gamma q^2,\beta)$-Furstenberg set then there exists a subset $\cF\subseteq \cL^*_2(\F_q^n)$ of size at least $\beta|\cL^*_2(\F_q^n)|$ such that for every $A\in \cF$ there exists a $tq=\lceil \gamma q^2\rceil$-rich shift $a+A$ for some $a\in \F_q^n$. To $\cF$ we can also associate a set of elements $$\cF'=\{ut_1+vt_2|u,v\in \F_q^n, \spanV\{u,v\}\in \cF\}\subseteq \F_q(t_1,t_2)^n \subset \cV_{\text{full}}.$$ 

Note, in general for each flat $A\in \cL^*_2(\F^n_q)$ there are $(q^2-1)(q^2-q)$ elements in $\cV_{\text{full}}$ corresponding to it (because there are $(q^2-1)(q^2-q)$ ordered pairs of vectors which span $A$) and each element $ut_1+vt_2\in \cV_{\text{full}}$ corresponds to a unique choice of basis. Thus, we have $$|\cF'|\ge \beta|\cV_{\text{full}}|.$$ 

Let  $\ell$ be an integer parameter (we will later send $\ell$ to infinity) and take
$$m=(q^2+t-1)\ell$$
and $$d=q^2t\ell-1$$
for $\ell\in \N$. As $d<q^2t\ell$, using Corollary \ref{cor:goodMono} we can find a set $P_{\cF'}(d,t\ell)$ of monomials of degree at most $d$ so that $$|P_{\cF'}(d,t\ell)| \geq \beta\binom{d+n}{n}$$ and such that no $\F_q$-linear combination of monomials in $P_{\cF'}(d,t\ell)$ vanishes over all points in~$\cF'$ with multiplicity at least $t\ell$.   
If 
$$\beta\frac{\binom{d+n}{n}}{\binom{m+n-1}{n}}> |K|,$$ 
then we can find (by solving a system of homogeneous linear equations) a non-zero polynomial $Q\in \F_q[x_1,\hdots,x_n]$ of degree at most $d$ spanned by monomials in $P_{\cF'}(d,t\ell)$ vanishing with multiplicity at least $m$ on every point in $K$. Let $Q^H$ be the highest degree homogenous part of $Q$.

\begin{claim}
For any $x\in \cF'$ we have
$\text{mult}(Q^H,x)\ge t\ell.$
\end{claim}
\begin{proof}
    Let $\bj\in \Z_{\ge 0}^n$ be such that $\text{wt}(\bj)<t\ell$. By Lemma \ref{lem:chainRule}, $Q^{(\bj)}$ vanishes on $K$ with multiplicity at least $m-\text{wt}(\bj)$. $Q^{(\bj)}$ is also of degree at most $d-\text{wt}(\bj)$
    
    By construction for every element $ut_1+vt_2\in \cF'$ there exists an element $c_{u,v}\in \F_q^n$ such that $c_{u,v}+\{ut_1+vt_2|t_1,t_2\in \F_q\}$ is $qt$-rich with respect to $K$. By Lemma \ref{lem:multComp} we have that the bivariate polynomial  $Q^{(\bj)}(c_{u,v}+ut_1+vt_2)$ vanishes on $qt$ many points in $ \F_q^2$ with multiplicity at least $m-\text{wt}(\bj)$.
    
    By Lemma \ref{multSchwartz}, we have that,
    if $Q^{(\bj)}(c_{u,v}+ut_1+vt_2)$ is non-zero then,
    $$(d-\text{wt}(\bj))q\ge (m-\text{wt}(\bj))qt.$$
    Rearranging gives us:
    $$d+\text{wt}(\bj)(t-1)\ge mt.$$
    Substituting $d=q^2t\ell-1$, $m=(q^2+t-1)\ell$ and using the fact that $\text{wt}(\bj)<t\ell$ gives us:    
    $$q^2t\ell-1 +(t-1)t\ell > q^2t\ell+(t-1)t\ell.$$
    This leads to a contradiction. This means that $Q^{(\bj)}(c_{u,v}+ut_1+vt_2)$ is identically $0$. We note, $Q^{(\bj)}(c_{u,v}+ut_1+vt_2)\in \F_q[t_1,t_2]$ and its highest degree homogeneous part is $(Q^H)^{(\bj)}(ut_1+vt_2)$. This means $(Q^H)^{(\bj)}(ut_1+vt_2)=0$ for all $\bj$ such that $\text{wt}(\bj)<t\ell$. This proves the claim. 
\end{proof}

As $Q^H$ is a non-zero polynomial with coefficients in $\F_q$ spanned  by monomials in $P_{\cF'}(d,t\ell)$ and it vanishes with multiplicity at least $t\ell$ on every point in $\cF'$, we get a contradiction to Corollary \ref{cor:goodMono}. Therefore, we can conclude that
$$\beta\frac{\binom{d+n}{n}}{\binom{m+n-1}{n}}\le |K|.$$
Substituting $d=q^2t\ell-1$, $m=(q^2+t-1)\ell$ gives us:
$$|K|\ge \beta\frac{(q^2t\ell-1+n)(q^2t\ell-2+n)\hdots(q^2t\ell)}{((q^2+t-1)\ell+n-1)((q^2+t-1)\ell+n-2)\hdots ((q^2+t-1)\ell)}\,. $$
Letting $\ell\rightarrow \infty$ gives us:
$$|K|\ge \beta t^n \left(1+\frac{t-1}{q^2}\right)^{-n}\,.$$
As $q\ge t\ge \gamma q$ the proof of the lemma is complete.
\end{proof}

We note the arguments in this section easily generalizes for $(k,\gamma q^k,\beta)$-Furstenberg sets for all $k\ge 1$ to prove the following theorem.

\begin{theorem}[Size of $(k,\gamma q^k,\beta)$-Furstenberg Sets]
For any $\gamma \in [0,1],\beta\in [0,1], n\in \N$, $q$ a prime power every $(k,\gamma q^k,\beta)$-Furstenberg set $K\subseteq \F_q^n$ has size at least,
$$|K|\ge \beta\gamma^n q^n\left(1+\frac{1}{q^{k-1}}\right)^{-n}.$$
\end{theorem}
The \cite{ORW22} reduction gives a quantitatively worse bound of $\beta\gamma^n q^n (2^{n}\log_2(2en)e)^{-1}$ compared to $\beta\gamma^n q^n 2^{-n}$ for $k=1$. Note that for $k\ge 5$, Theorem~\ref{thm-existsSB} gives us much better bounds for $q>> n$.

\printbibliography

@inproceedings{MarAngeBalls,
author="Raab, Martin
and Steger, Angelika",
title="``{B}alls into {B}ins'' - {A} {S}imple and {T}ight {A}nalysis",
doi="10.1007/3-540-49543-6_13",
booktitle="Randomization and Approximation Techniques in Computer Science (RANDOM)",
year="1998",
publisher="Springer",
pages="159-170",
isbn="978-3-540-49543-7"
}

@article{ALONrandomized,
title = {A fast and simple randomized parallel algorithm for the maximal independent set problem},
journal = {Journal of Algorithms},
volume = {7},
number = {4},
pages = {567-583},
year = {1986},
issn = {0196-6774},
doi = {10.1016/0196-6774(86)90019-2},
author = {Noga Alon and László Babai and Alon Itai}
}

@article{ORW22,
title = "{New bounds on the density of lattice coverings}",
author = "Or Ordentlich and Oded Regev and Barak Weiss",
year = "2022",
doi = "10.1090/JAMS/984",
volume = "35",
pages = "295-308",
journal = "Journal of the American Mathematical Society",
issn = "0894-0347",
publisher = "American Mathematical Society",
number = "1",
}

@article{KLSS2011,
  title={{Kakeya-type sets in finite vector spaces}},
  author={Kopparty, Swastik and Lev, Vsevolod F. and Saraf, Shubhangi and Sudan, Madhu},
  journal={Journal of Algebraic Combinatorics},
  volume={34},
  doi="10.1007/S10801-011-0274-8",
  number={3},
  pages={337-355},
  year={2011},
  publisher={Springer}
}

@inproceedings{StatZeroCrypto,
author = {Kaslasi, Inbar and Rothblum, Ron D. and Vasudevanr, Prashant Nalini},
title = {{Public-Coin Statistical Zero-Knowledge Batch Verification Against Malicious Verifiers}},
year = {2021},
isbn = {978-3-030-77882-8},
publisher = {Springer},
url = {https://doi.org/10.1007/978-3-030-77883-5_8},
doi = {10.1007/978-3-030-77883-5_8},
booktitle = {Advances in Cryptology: {EUROCRYPT} 2021, {Part} III},
pages = {219–246},
numpages = {28}
}

@BOOK{goldreich_2008,
  TITLE = {{Computational Complexity: A Conceptual Perspective}},
doi="10.1145/1412700.1412710",
  AUTHOR = {Goldreich, Oded},
  YEAR = {2008}, 
  PUBLISHER = {Cambridge University Press}
}

@InProceedings{dictionaryCrypt,
author="Canetti, Ran
and Halevi, Shai
and Steiner, Michael",
title="{Mitigating Dictionary Attacks on Password-Protected Local Storage}",
booktitle="Advances in Cryptology - CRYPTO 2006",
year="2006",
doi="10.1007/11818175_10",
publisher="Springer",
pages="160-179",
isbn="978-3-540-37433-6"
}

@article{dhar2021kakeya,
	journal={Advances in Combinatorics},
	doi={10.19086/aic.2024.2},
	title={The {K}akeya Set conjecture over {$\mathbb{Z}$/N$\mathbb{Z}$} for general {N}},
	author={Dhar, Manik},
	date={2024-01-26},
	year=2024,
	month=1,
	day=26,
}

@unpublished{dhar2022maximal,
      title={{Maximal Kakeya and $(m,\epsilon)$-Kakeya bounds over $\mathbb{Z}/\text{{N}}\mathbb{Z}$ for general N}}, 
      author={Manik Dhar},
year={2022},
doi={10.48550/arXiv.2209.11443},
note = {arxiv preprint},
}

@article{JaikumarShma2000,
author = {Radhakrishnan, Jaikumar and Ta-Shma, Amnon},
title = {{Bounds for Dispersers, Extractors, and Depth-Two Superconcentrators}},
journal = {SIAM Journal on Discrete Mathematics},
volume = {13},
number = {1},
pages = {2-24},
year = {2000},
doi = {10.1137/S0895480197329508},
URL = {https://doi.org/10.1137/S0895480197329508},
eprint = {https://doi.org/10.1137/S0895480197329508}
}

@article{green2019arithmetic,
  title={On the arithmetic {K}akeya conjecture of {K}atz and {T}ao},
  doi="10.1007/s10998-018-0270-z",
  author={Green, Ben and Ruzsa, Imre Z},
  journal={Periodica Mathematica Hungarica},
  volume={78},
  number={2},
  pages={135-151},
  year={2019},
  publisher={Springer}
}

@article{mehlhorn1984randomized,
  title={{Randomized and deterministic simulations of PRAMs by parallel machines with restricted granularity of parallel memories}},
  author={Mehlhorn, Kurt and Vishkin, Uzi},
  journal={Acta Informatica},
  volume={21},
doi="10.1007/BF00264615",
  number={4},
  pages={339-374},
  year={1984},
  publisher={Springer}
}

@article{LinearNoga,
author = {Alon, Noga and Dietzfelbinger, Martin and Miltersen, Peter Bro and Petrank, Erez and Tardos, G\'{a}bor},
title = {{Linear Hash Functions}},
year = {1999},
publisher = {Association for Computing Machinery},
volume = {46},
number = {5},
issn = {0004-5411},
url = {https://doi.org/10.1145/324133.324179},
doi = {10.1145/324133.324179},
journal = {Journal of Association for Computing Machinery},
pages = {667–683},
numpages = {17}
}

@article{CARTERWEGMANS,
title = {{Universal classes of hash functions}},
journal = {Journal of Computer and System Sciences},
volume = {18},
number = {2},
pages = {143-154},
year = {1979},
issn = {0022-0000},
doi = {10.1016/0022-0000(79)90044-8},
url = {https://www.sciencedirect.com/science/article/pii/0022000079900448},
author = {J. Lawrence Carter and Mark N. Wegman}
}

@InProceedings{LHLRevisited,
author="Barak, Boaz
and Dodis, Yevgeniy
and Krawczyk, Hugo
and Pereira, Olivier
and Pietrzak, Krzysztof
and Standaert, Fran{\c{c}}ois-Xavier
and Yu, Yu",
title="{Leftover Hash Lemma, Revisited}",
doi="10.1007/978-3-642-22792-9_1",
booktitle="Advances in Cryptology - CRYPTO 2011",
year="2011",
publisher="Springer",
pages="1-20",
isbn="978-3-642-22792-9"
}

@inproceedings{ILL,
author = {Impagliazzo, Russell and Levin, Leonid A. and Luby, Michael},
title = {{Pseudo-Random Generation from One-Way Functions}},
year = {1989},
isbn = {0897913078},
publisher = {Association for Computing Machinery},
url = {https://doi.org/10.1145/73007.73009},
doi = {10.1145/73007.73009},
pages = {12–24},
numpages = {13},
booktitle = {Proceedings of the Twenty-First Annual ACM Symposium on Theory of Computing},
series = {STOC '89}
}

@article{DKSS13,
author = {Dvir, Zeev and Kopparty, Swastik and Saraf, Shubhangi and Sudan, Madhu},
title = {{Extensions to the Method of Multiplicities, with Applications to Kakeya Sets and Mergers}},
journal = {SIAM Journal on Computing},
volume = {42},
number = {6},
pages = {2305-2328},
year = {2013},
doi = {10.1137/100783704},

URL = { 
        https://doi.org/10.1137/100783704
    
},
eprint = { 
        https://doi.org/10.1137/100783704
    
}
}

@ARTICLE{BukhChao21,
  AUTHOR = { Boris Bukh,Ting-Wei Chao
},
  TITLE = {{Sharp density bounds on the finite field Kakeya problem}},
  JOURNAL = {Discrete Analysis},
volume={26},
  doi = {10.19086/da.30707},
  YEAR = 2021
   }

@article{DW08,
author = {Dvir, Zeev and Wigderson, Avi},
title = {Kakeya Sets, New Mergers, and Old Extractors},
journal = {SIAM Journal on Computing},
volume = {40},
number = {3},
pages = {778-792},
year = {2011},
doi = {10.1137/090748731}}

@article{Wolff99,
author = "Thomas Wolff",
title = "Recent work connected with the {K}akeya problem",
journal = "Prospects in mathematics (Princeton, NJ, 1996)",
pages = "129-162",
doi="10.1090/ulect/029/11",
Publisher = "American Mathematical Society",
year = "1999"
}

@article{Dvir08,
  author = {Zeev Dvir},
  title = {{On the size of {K}akeya sets in finite fields}},
  journal = {Journal of American Mathematical Society},
  volume = {22},
  doi={10.1090/S0894-0347-08-00607-3},
  year = {2009},
  pages = {1093-1097},
}

@inproceedings{schwartz1979probabilistic,
  title={{Probabilistic algorithms for verification of polynomial identities}},
  author={Schwartz, Jacob T.},
doi="10.1145/322217.322225",
  booktitle={Symbolic and Algebraic Computation (EUROSAM 1979)},
  pages={200-215},
  year={1979},
  organization={Springer}
}

@InProceedings{ZippelPaper,
author="Zippel, Richard",
editor="Ng, Edward W.",
title="{Probabilistic algorithms for sparse polynomials}",
booktitle="Symbolic and Algebraic Computation (EUROSAM 1979)",
year="1979",
doi="10.1007/3-540-09519-5_73",
publisher="Springer",
pages="216-226",
isbn="978-3-540-35128-3"
}

@article{DDL-1,
title={Furstenberg Sets in Finite Fields: Explaining and Improving the {E}llenberg–{E}rman Proof},
  author={Manik Dhar and Zeev Dvir and Ben D. Lund},
  journal={Discrete \& Computational Geometry},
  year={2019},
doi={10.1007/s00454-023-00585-y},
  pages={1-31},
}

@article{EE16,
title={Furstenberg sets and {F}urstenberg schemes over finite fields},
  author={Ellenberg, Jordan and Erman, Daniel},
  journal={Algebra \& Number Theory},
  volume={10},
  number={7},
  pages={1415-1436},
doi="10.2140/ant.2016.10.1415",
  year={2016},
  publisher={Mathematical Sciences Publishers}
}

@article{DDL-2,
author = {Manik Dhar and Zeev Dvir and Ben Lund},
     title = {{Simple proofs for Furstenberg sets over finite fields}},
     journal = {Discrete Analysis},
     year = {2021},
     number = {22},
     doi = {10.19086/da.29067},
     language = {en},
     url = {https://discreteanalysisjournal.com/article/29067}
}

@article{EOT10,
  title={The {K}akeya set and maximal conjectures for algebraic varieties over finite fields},
  author={Ellenberg, Jordan S and Oberlin, Richard and Tao, Terence},
  journal={Mathematika},
  volume={56},
  number={1},
  pages={1-25},
  doi="10.1112/S0025579309000400",
  year={2010},
  publisher={London Mathematical Society}
}

\end{document}